\documentclass[11pt]{amsart}
\usepackage[utf8]{inputenc}
\usepackage[a4paper,top=3cm,bottom=3cm,left=3cm,right=3cm,marginparwidth=1.5cm]{geometry}
\usepackage[british]{babel}
\usepackage{graphicx}
\usepackage{url}
\usepackage{amsmath}
\usepackage{mathtools}
\usepackage{mwe}
\usepackage{amssymb}
\usepackage{float}
\usepackage{eucal} 
\usepackage{amsthm}
\usepackage{mathtools}
\usepackage{xcolor}
\usepackage{amsfonts}
\usepackage{hyperref}
\usepackage[nameinlink]{cleveref}
\usepackage{bbm} 
\usepackage{adjustbox} 
\usepackage{enumitem} 
\setlist[enumerate]{ label=(\alph*) }
\usepackage{quiver}

\usepackage[backend=biber,style=alphabetic, giveninits=false, urldate=long, sorting=nyt, useprefix=true, url=false, doi=false, isbn=false, maxalphanames=99, maxbibnames=99]{biblatex}
\DeclareFieldFormat[article,incollection, misc, book, phdthesis, inproceedings]{title}{#1}
\renewbibmacro{in:}{}

\DeclareMathOperator{\Z}{\mathbb{Z}}

\DeclareMathOperator{\im}{Im}
\DeclareMathOperator{\Ker}{Ker}

\DeclareMathOperator{\id}{id}
\DeclareMathOperator{\Spec}{Spec}
\DeclareMathOperator{\Spc}{Spc}
\DeclareMathOperator{\thick}{thick}
\DeclareMathOperator{\Hom}{Hom}

\DeclareMathOperator{\fgmod}{mod}
\DeclareMathOperator{\supp}{supp}

\DeclareMathOperator{\Loc}{Loc}
\DeclareMathOperator{\op}{op}
\DeclareMathOperator{\gen}{gen}
\DeclareMathOperator{\unit}{\mathbbm{1}}

\DeclareMathOperator{\D}{D}
\DeclareMathOperator{\perf}{perf}
\DeclareMathOperator{\K}{K}

\DeclareMathOperator{\fp}{fp}

\DeclareMathOperator{\SH}{SH}

\DeclareMathOperator{\stab}{stab}
\DeclareMathOperator{\Stab}{Stab}
\DeclareMathOperator{\perm}{perm}
\DeclareMathOperator{\Res}{Res}
\DeclareMathOperator{\Ind}{Ind}
\DeclareMathOperator{\DPerm}{DPerm}
\DeclareMathOperator{\Infl}{Infl}
\DeclareMathOperator{\Inj}{Inj}
\DeclareMathOperator{\Proj}{Proj}

\newcommand{\mc}{\mathcal}
\newcommand{\mf}{\mathfrak}
\newcommand{\mb}{\mathbb}

\newcommand{\Mod}{\mathrm{Mod}}
\newcommand{\overbar}[1]{\mkern 1.5mu\overline{\mkern-3mu#1\mkern-3mu}\mkern 1.5mu}

\numberwithin{equation}{section}

\theoremstyle{plain}
\newtheorem{theorem}[equation]{Theorem}
\newtheorem*{theorem*}{Theorem}
\newtheorem{proposition}[equation]{Proposition}
\newtheorem{lemma}[equation]{Lemma}
\newtheorem{corollary}[equation]{Corollary}

\AddToHook{env/theorem/begin}{\crefalias{equation}{theorem}}
\AddToHook{env/proposition/begin}{\crefalias{equation}{proposition}}
\AddToHook{env/lemma/begin}{\crefalias{equation}{lemma}}
\AddToHook{env/corollary/begin}{\crefalias{equation}{corollary}}

\theoremstyle{definition}
\newtheorem{definition}[equation]{Definition}
\newtheorem*{definition*}{Definition}

\newtheorem{example}[equation]{Example}
\newtheorem{notation}[equation]{Notation}

\AddToHook{env/definition/begin}{\crefalias{equation}{definition}}
\AddToHook{env/notation/begin}{\crefalias{equation}{notation}}
\AddToHook{env/example/begin}{\crefalias{equation}{example}}
\AddToHook{env/examples/begin}{\crefalias{equation}{examples}}

\theoremstyle{remark}
\newtheorem{hypothesis}[equation]{Hypothesis}
\newtheorem{construction}[equation]{Construction}
\newtheorem{recollection}[equation]{Recollection}
\newtheorem{remark}[equation]{Remark}

\crefname{recollection}{Recollection}{Recollections}

\AddToHook{env/hypothesis/begin}{\crefalias{equation}{hypothesis}}
\AddToHook{env/construction/begin}{\crefalias{equation}{construction}}
\AddToHook{env/recollection/begin}{\crefalias{equation}{recollection}}
\AddToHook{env/remark/begin}{\crefalias{equation}{remark}}
\AddToHook{env/remarks/begin}{\crefalias{equation}{remarks}}
\AddToHook{env/conventions/begin}{\crefalias{equation}{conventions}}

\title{Residual regularity in tensor triangular geometry}

\author{Emmy Van Rooy}
\address{Emmy Van Rooy, UCLA Mathematics Department, Los Angeles, CA 90095-1555, USA}
\email{emmyvr@ucla.edu}
\urladdr{http://www.math.ucla.edu/~emmyvr}
\date{\today}

\begin{document}
\begin{abstract}
   We investigate a new notion of regularity for tensor triangulated categories, called residual regularity. We show that residual regularity descends and ascends via finite separable extensions and we classify all finite groups whose derived category of permutation modules is residually regular. 
\end{abstract}

\maketitle
\tableofcontents

\section*{Introduction}

In \cite{BalmerOG}, Balmer associates to any essentially small tensor triangulated category $\mc K$ a geometric space $\Spc(\mc K)$ in a universal way. This space $\Spc(\mc K)$ has since been referred to as the \emph{Balmer spectrum} of $\mc K$, and the resulting theory is named tensor triangular geometry --- tt-geometry for short. Since then, tt-geometry has become a well-established technique to study tensor triangulated categories (abbreviated similarly as tt-categories). Among many other things, tt-geometry allows one to classify objects in the essentially small tt-category $\mc K$ that are the same up to the basic operations available in~$\mc K$: taking suspensions, cones, direct sums and summands or tensoring with objects. In fact, this classification is equivalent to the computation of the Balmer spectrum $\Spc(\mc K)$. 
Furthermore, the theory covers a large class of specialised, and remarkably distinct, areas of mathematics, ranging from algebra to analysis. To name a few, it applies to algebraic geometry,  
stable homotopy theory, modular representation theory, motivic theory and noncommutative topology.
In this way, tt-geometry gives a unifying approach to many fields of mathematics.
We refer to the surveys \cite{ICMpaper, homotopyhandbook, Stevensontour} for an overview of the basic theory, its virtues and its applications.

Often, the essentially small tt-category $\mc K$ sits inside a larger, but still orderly, tt-category $\mc T$. For example, think of perfect complexes $\D^{\perf}(R)$ over a commutative ring $R$ inside the unbounded derived category $\D(R)$; or the Spanier-Whitehead category of finite pointed CW-complexes $\SH^{\text{fin}}$ inside the stable homotopy category of topological spectra $\SH$. More precisely, $\mc T$ is a \emph{rigidly-compactly generated tt-category} and $\mc K = \vcentcolon \mc T^c$ is its subcategory of rigid-compact objects. Such a tt-category $\mc T$ is referred to as a `big' tt-category.  
The subcategory $\mc T^c$ is always an essentially small tt-category, so we can consider its Balmer spectrum $\Spc(\mc T^c)$. 
Approaching the study of $\mc T$ through the tt-geometry of $\mc T^c$ is a rich topic, and this is precisely the context that we will be in today.

A common theme in tt-geometry is to transpose ideas and techniques between the various fields it brings together.  
Taking inspiration from algebraic geometry in particular has often proved fruitful. In algebraic geometry, \emph{regularity} is a well-studied condition on Noetherian schemes, and regular points are, in a precise sense, better behaved than non-regular (i.e.\@ singular) ones. 
Let us, with the power of foresight, recall the following characterisation of regularity for a commutative Noetherian ring $R$: a point $\mathfrak{p} \in \Spec(R)$ is regular if and only if the residue field $R_\mathfrak{p}/\mathfrak{p}R_\mathfrak{p}$ is perfect over the local ring $R_{\mathfrak{p}}$, i.e.\@ it is a compact object in the category $\D(R_{\mathfrak{p}})$.
In this work we investigate a notion of regularity for `big' tt-categories, dubbed \emph{residual regularity}, which mirrors this characterisation of classical regularity in commutative algebra. 

The road to residual regularity is paved by a type of \emph{residue objects} that are afforded to us through the framework of \emph{homological residue fields} \cite{BKSruminations}. These homological residue fields are homological symmetric monoidal functors 
\begin{equation*}
    \mathbbm{h}_{\mc H} \colon \mc T \to \mc A(\mc T)_{\mc H},
\end{equation*}
from our `big' tt-category $\mc T$ to various tensor abelian categories $\mc A(\mc T)_{\mc H}$. The parameter $\mc H$ ranges over a set called the \emph{homological spectrum} of $\mc T$, which is denoted by $\Spc^{\mathrm{h}}(\mc T^c)$. The homological spectrum is very close to the Balmer spectrum in the sense that there is a canonical map
\begin{equation*}
    \pi_{\mc T} \colon \Spc^{\mathrm{h}}(\mc T^c) \to \Spc(\mc T^c),
\end{equation*}
that is always surjective and is known to be bijective for many standard tt-categories across algebraic geometry, modular representation theory and stable homotopy theory. However, it should be noted that it can fail to be injective, as was recently shown in \cite{nervesofsteelcounterex}. The theory of homological residue fields has already shown its mettle: they give rise to a tensor triangular nilpotence theorem \cite{Balmernilpotence} and a support theory for non-compact objects of `big' tt-categories \cite{bigsupport}. 

The homological residue fields furnish us with the aforementioned residue objects as follows: in $\mc A(\mc T)_{\mc H}$ the tensor unit $\overline{\unit}$ admits an injective hull $\mathbbm{h}_{\mc H}(\kappa(\mc H))$ that comes from a unique object $\kappa(\mc H)$ in $\mc T$ via the functor~$\mathbbm{h}_{\mc H}$. These residue objects $\kappa(\mc H)$ are in general \emph{non-compact} objects of $\mc T$. They play the role of the classical residue fields in commutative algebra and indeed recapture exactly those in the case of the derived category of a commutative ring. 
The tt-analogue of regularity presents itself: 
\begin{definition*}
    A `big' tt-category $\mc T$ is called \emph{residually regular} if for every point $\mc P \in \Spc(\mc T^c)$, there exists an $\mc H \in \pi_{\mc T}^{-1}(\{\mc P\})$ such that the residue object $\kappa(\mc H)$ is a compact object in the local category $\mc T_{\mc P} \vcentcolon= \mc T/\Loc(\mc P)$. 
\end{definition*}
This definition was proposed by Paul Balmer in private communication. It follows immediately that residual regularity can be checked locally on the Balmer spectrum. Furthermore, by construction, a Noetherian scheme $X$ is regular if and only if its `big' derived category $\D(X)= \D_{\mathrm{Qcoh}}(X)$ is residually regular. Using the extensive theory of homological residue fields, one can show that residual regularity of $\mc T$ implies that the map $\pi_{\mc T}$ is in fact a bijection and that the Balmer spectrum $\Spc(\mc T^c)$ is a \emph{weakly Noetherian} topological space. This already excludes the stable homotopy category of spectra $\SH$ from being residually regular, as its Balmer spectrum is not weakly Noetherian; see \Cref{stable homotopy cat}.

There have been various efforts in the literature to generalise regularity to other settings, such as e.g.\@ \cite{singcats,BIKP, Krausecompletion, kostas, Neeman2026, Orlov, Rouquier} for triangulated categories and \cite{BarwickLawson, GregGreenlees} for differential graded algebras and structured ring spectra. This is also intimately related to the search for an analogue of the bounded derived category of a Noetherian scheme for more general triangulated categories; see \cite{Neeman2026, Rouquier}. We return to this topic in \Cref{inthelit}, but the definitive relationship of these other notions to residual regularity will require further investigation.

In modular representation theory, the most well-known `big' tt-categories are arguably the stable module category $\Stab(kG)$ and the homotopy category of injectives $\K\Inj(kG)$, for a finite group $G$ and a field $k$. In recent years, another actor has entered the scene, the \emph{derived category of permutation $kG$-modules}, denoted by $\DPerm(G;k)$. 
Permutation $kG$-modules are easy to grasp: they are simply the $k$-linearisations of $G$-sets. Informally then, the `big' tt-category $\DPerm(G;k)$ is a certain localisation of the homotopy category of all permutation $kG$-modules, and it admits the categories $\K\Inj(kG)$ and $\Stab(kG)$ as further localisations. Permutation modules are closely related to equivariant homotopy theory, Mackey functors and Artin motives; see \cite{BGconnections}. The tt-geometry of these categories $\DPerm(G;k)$ has been studied extensively by Balmer--Gallauer in \cite{TT-perm}.

Our main result classifies all finite groups whose derived category of permutation modules is residually regular: 
\begin{theorem*}[\Cref{modular rep theorem}]
    Let $G$ be a finite group and $k$ a field of positive characteristic $p$ dividing the order of $G$. The derived category of permutation $kG$-modules $\DPerm(G;k)$ is residually regular if and only if $p=2$ and all Sylow $p$-subgroups of $G$ are cyclic. 
\end{theorem*}
For the worried reader, we mention that none of the points at which we prove that residual regularity of $\DPerm(G;k)$ fails `come from' $\Stab(kG)$ (or $\K\Inj(kG)$ for that matter); all points of $\Stab(kG)$ are either residually regular or unknown. Indeed, as $\Stab(kG)^c$ is a \emph{strongly generated} triangulated category, it is the expectation that $\Stab(kG)$ is residually regular for any choice of finite group $G$ and field $k$.

A core part of the proof of this classification is a technical result concerning \emph{separable extensions of tt-categories}, as introduced in \cite{tt-rings}. 
These are special kinds of functors between tt-categories which can be thought of as an analogue of étale extensions in algebraic geometry. Separable extensions are ubiquitous in tt-geometry. In modular representation theory, the restriction to any subgroup gives rise to a separable extension at the level of the derived category of permutation modules. The above classification follows grosso modo by doing some hands-on computations for specific small groups and then applying the following result to the separable extensions induced by restriction to these groups: 
\begin{theorem*}[\Cref{the big one}]
    Let $F \colon \mc T \to \mc S$ be a finite separable extension of finite degree and let $\mc Q \in \Spc(\mc S^c)$. Suppose that $\pi_{\mc S} \colon \Spc^{\mathrm{h}}(\mc S^c) \to \Spc(\mc S^c)$ is a bijection and that $\Spc(\mc S^c)$ is Noetherian. Then $\mc S$ is residually regular at the point $\mc Q$ if and only if $\mc T$ is residually regular at the point $ \mc P \vcentcolon= \varphi_{F}(\mc Q)$, where $\varphi_{F} \colon \Spc(\mc S^c) \to \Spc(\mc T^c)$.  
\end{theorem*}
We refer to the body of the paper for the unexplained hypotheses in this result, and simply remark that these are relatively mild conditions. Heuristically, the result constitutes a descent-ascent statement for residual regularity that can be viewed as a tt-analogue of (classical) regularity being an étale local property in algebraic geometry. In fact, the descent part (the left-to-right direction) in this descent-ascent statement already holds under weaker conditions; see \Cref{left to right}. 

Finally, we also point out another result with a counterpart in algebraic geometry. When $R$ is a commutative Noetherian ring that is regular at a closed point $\mathfrak{m} \in \Spec(R)$ (i.e.\@ $\mathfrak{m}$ is a maximal ideal of $R$), then the residue field $R_\mathfrak{m}/\mathfrak{m}R_\mathfrak{m}$ is a perfect complex already in $\D(R)$, not just in the stalk category~$\D(R_\mathfrak{m})$. In \Cref{residue regularity at closed points}, we show that the same is true in tt-geometry for the residue objects of closed points, under the condition that the closed point satisfies an additional topological hypothesis, which is automatically satisfied when the spectrum $\Spc(\mc T^c)$ is Noetherian.

The outline of the paper is relatively straightforward.  
In \Cref{homological residue fields sec} we recall the relevant details of the theory of homological residue fields. Thereafter, \Cref{res reg sec} is dedicated to the definition of residual regularity and its first consequences, among which is the descent statement of the above descent-ascent result. Next, in \Cref{descent stat sec}, we start by recalling the theory of separable rings in tensor triangulated and tensor abelian categories, and then prove the ascent statement in a series of technical lemmas. Finally, in \Cref{mod rep sec}, we prove the classification of all finite groups whose derived category of permutation modules is residually regular. 

\subsection*{Acknowledgements}
I am very grateful to my advisor Paul Balmer for his valuable ideas, guidance and useful comments. I would also like to thank Greg Stevenson for helpful discussions regarding separable extensions of tensor abelian categories.
Moreover, I thank Timothy De Deyn, Beren Sanders and Greg Stevenson for their feedback on an earlier version of this paper.

\section*{Conventions}

We often write `tt' to abbreviate `tensor triangulated' or `tensor triangular'.

Throughout, a `big' tt-category refers to a rigidly-compactly generated tensor triangulated category $\mc T$, as in \cite{Balmer-Favi}. This means that $\mc T$ admits all set-indexed coproducts, its compact objects and its rigid objects coincide, and the essentially small subcategory $\mc T^c$ of compact-rigid objects generates $\mc T$ as a localising triangulated subcategory. The tensor product of $\mc T$ will be denoted by $\otimes$, the unit by $\unit$ and the suspension functor by $\Sigma$. The tt-category $\mc T^c$ is essentially small, so we can consider its Balmer spectrum $\Spc(\mc T^c)$. We will write $\supp_{\mc T}(-)$ for the universal support theory on $\mc T^c$, i.e.\@ 
\begin{equation*}
    \supp_{\mc T}(a) = \{ \mc P \in \Spc(\mc T^c) \mid a \notin \mc P \},
\end{equation*}
for $a \in \mc T^c$.

A tensor triangulated functor $F\colon \mc T \to \mc S$ between two `big' tt-categories is called geometric if it preserves all coproducts. By Brown--Neeman Representability \cite{Neemanrepresentability}, any geometric tt-functor admits a right adjoint $U \colon\mc S \to \mc T$, which is triangulated, lax-monoidal and preserves coproducts. Moreover, the adjunction $F \dashv U$ satisfies a projection formula: $U(X \otimes F(Y)) \cong U(X) \otimes Y$ for objects $X \in \mc S$ and $Y \in \mc T$. Since $F$ is symmetric monoidal, it preserves rigid objects and therefore restricts to a tt-functor $F\colon \mc T^c \to \mc S^c$. Thus, $F$ induces a continuous map on Balmer spectra that we denote by $\varphi_F\colon \Spc(\mc S^c) \to \Spc(\mc T^c)$. 

For a collection of objects $S \subseteq \mc T$, we denote by $\thick(S)$, respectively $\Loc(S)$, the thick triangulated subcategory generated by $S$, respectively localising triangulated subcategory generated by $S$. By contrast, the thick $\otimes$-ideal generated by $S$ is denoted by $\langle S \rangle$.

Analogously, but by abuse of notation, when $\mc A$ is a Grothendieck abelian category with colimit-preserving symmetric monoidal structure and $S\subseteq \mc A$ is a collection of objects, we denote by $\Loc(S)$, respectively $\langle S\rangle$, the localising Serre subcategory, respectively Serre $\otimes$-ideal generated by $S$.

\section{Homological residue fields}\label{homological residue fields sec}

We recall the relevant details of the theory of homological residue fields.
Experts can safely skip this section and refer back to it as needed.

\begin{recollection}
    For a `big' tt-category $\mc T$, the module category $\mc A(\mc T) \vcentcolon= \Mod\text{-}\mc T^c$ is the locally coherent Grothendieck category of additive functors from $(\mc T^c)^{\op}$ to abelian groups (see \cite[Appendix A]{BKSruminations}). The subcategory of finitely presented objects of $\mc A(\mc T)$ is denoted by $\mc A(\mc T)^{\fp}$ and it is abelian, i.e.\@ $\mc A(\mc T)$ is locally coherent. The restricted Yoneda functor $h_{\mc T}\colon \mc T \to \mc A(\mc T)$ is defined by $h(X) =\hat{X} = \Hom_{\mc T}(-, X)_{\mid \mc T^c}$ for every $X \in \mc T$. This functor is homological, coproduct-preserving and conservative. It yields a commutative diagram 
\[\begin{tikzcd}
	{\mc T^c} & {\mc A(\mc T)^{\fp}} \\
	{\mc T} & {\mc A(\mc T) \rlap{ ,}}
	\arrow[hook, from=1-1, to=1-2]
	\arrow[hook, from=1-1, to=2-1]
	\arrow[hook, from=1-2, to=2-2]
	\arrow["{h_{\mc T}}"', from=2-1, to=2-2]
\end{tikzcd}\]
where the top arrow is the classical Yoneda embedding for $\mc T^c$, which identifies $\mc T^c$ with the finitely presented projective objects in $\mc A(\mc T)$. When the category $\mc T$ is clear from context, we sometimes drop the subscript and simply write $h$ for $h_{\mc T}$. 

By Day convolution, the module category $\mc A(\mc T)$ admits a symmetric monoidal structure ${\otimes: \mc A(\mc T) \times \mc A(\mc T) \to \mc A(\mc T)}$ that is colimit-preserving in each variable. This makes $h_{\mc T}$ into a symmetric monoidal functor, i.e.\@ $\widehat{X \otimes Y} \cong \hat{X} \otimes \hat{Y}$ for all $X, Y \in \mc T$, and $\hat{X}$ is $\otimes$-flat for every $X \in \mc T$, even non-compact. Furthermore, the tensor product descends to the finitely presented objects $\mc A(\mc T)^{\fp} \times \mc A(\mc T)^{\fp} \to \mc A(\mc T)^{\fp}$. See \cite[Appendix A]{BKSOG} for proofs of these claims.
\end{recollection}

\begin{recollection}\label{quotients of abelian cats}
    Let $\mc A$ be a locally coherent Grothendieck abelian category with a colimit-preserving symmetric monoidal structure and let $\mc I \subseteq \mc A^{\fp}$ be a Serre $\otimes$-ideal of the subcategory $\mc A^{\fp}$ of finitely presented objects. We write $\Loc(\mc I)$ for the smallest localising subcategory of $\mc A$ containing $\mc I$, which is then automatically a $\otimes$-ideal in $\mc A$ and we have $\mc I = \Loc(\mc I) \cap \mc A^{\fp}$. There is a Gabriel quotient 
\[\begin{tikzcd}
	{\mc A} \\
	{\mc A/\Loc(\mc I) }
	\arrow[""{name=0, anchor=center, inner sep=0}, "{Q_{\mc I}}"', shift right=2, from=1-1, to=2-1]
	\arrow[""{name=1, anchor=center, inner sep=0}, "{R_{\mc I}}"', shift right=2, from=2-1, to=1-1]
	\arrow["\dashv"{anchor=center}, draw=none, from=0, to=1]
\end{tikzcd}\]
where $Q_{\mc I}$ is the universal exact functor with kernel $\Loc(\mc I)$ and its right adjoint $R_{\mc I}$ is fully faithful. The Grothendieck abelian category $\mc A/\Loc(\mc I)$ is locally coherent and its subcategory of finitely presented objects can be identified with the quotient $\mc A^{\fp}/\mc I$. Furthermore, since $\Loc(\mc I)$ is a $\otimes$-ideal, $\mc A/\Loc(\mc I)$ inherits a unique symmetric monoidal structure such that $Q_{\mc I}$ is a symmetric monoidal functor. This monoidal structure is also colimit-preserving in each variable. See \cite[Appendix A]{BKSruminations} for justifications of these claims.

When $\mc A = \mc A(\mc T)$ is the module category of a `big' tt-category $\mc T$ and $\mc I \subseteq \mc A(\mc T)^{\fp}$, the composite $Q_{\mc I} \circ h_{\mc T}: \mc T \to \mc A(\mc T)/\Loc(\mc I)$ is coproduct-preserving, homological and symmetric monoidal. When the Serre $\otimes$-ideal $\mc I$ is clear from context, we often write 
\begin{equation*}
    \bar{X} \vcentcolon= (Q_{\mc I} \circ h_{\mc T})(X)=Q_{\mc I}(\hat{X}) \quad \text{and} \quad \bar{f}\vcentcolon=(Q_{\mc I} \circ h_{\mc T})(f)=Q_{\mc I}(\hat{f}),
\end{equation*}
for objects $X$ and morphisms $f$ in~$\mc T$. The object $\bar{X}$ remains $\otimes$-flat in $\mc A(\mc T)/\Loc(\mc I)$ for all $X \in \mc T$. Every injective object of $\mc A(\mc T)/\Loc(\mc I)$ is of the form $\bar{E}$ for some pure-injective object $E \in \mc T$, uniquely characterised by the property that $\hat{E} \cong R_{\mc I}(\bar{E})$. Moreover, the composite $Q_{\mc I} \circ h_{\mc T}$ induces, for every $X \in \mc T$, an isomorphism 
    \begin{equation}\label{injective iso}
        \Hom_{\mc T}(X, E) \xrightarrow{\sim} \Hom_{\mc A(\mc T)/{\Loc(\mc I)}}(\bar{X}, \bar{E}). 
    \end{equation}
    This follows from \cite{Krause}; see for example \cite[Corollary 2.18]{BKSruminations}. 
\end{recollection}

An important construction is the one of \cite[Section 3]{BKSruminations}: 
\begin{construction}\label{making the residue objects}
    For $\mc I \subseteq \mc A(\mc T)^{\fp}$ a Serre $\otimes$-ideal, by the isomorphisms \eqref{injective iso}, it follows that the injective hull of the tensor unit $\bar{\unit} \in \mc A(\mc T)/\Loc(\mc I)$ comes from $\mc T$ via $Q_{\mc I} \circ h_{\mc T}$: there exists a pure-injective object $E_{\mc I}$ and a morphism 
\begin{equation*}
    \eta_{\mc I}\colon \unit \to E_{\mc I}
\end{equation*} 
in $\mc T$ such that $\bar{\eta}_{\mc I}: \bar{\unit} \to \overbar{E_{\mc I}}$ is the injective hull in $\mc A(\mc T)/\Loc(\mc I)$. 
\end{construction}

\begin{recollection}\label{homological spectrum}
    Let $\mc T$ be a `big' tt-category. The \emph{homological spectrum} $\Spc^{\mathrm{h}}(\mc T^c)$ consists of all \emph{maximal} Serre $\otimes$-ideals $\mc H$ of~$\mc A(\mc T)^{\fp}$. Such an $\mc H$ is called a \emph{homological prime} of $\mc T$. We write $\mc A(\mc T)_{\mc H}\vcentcolon= \mc A(\mc T)/\Loc(\mc H)$ and $\mc A(\mc T)_{\mc H}^{\fp}\vcentcolon= \left( \mc A(\mc T)_{\mc H}\right)^{\fp} \cong \mc A(\mc T)^{\fp}/\mc H$ for the associated Gabriel quotients. Each homological prime $\mc H \in \Spc^{\mathrm{h}}(\mc T^c)$ yields a coproduct-preserving homological monoidal functor $Q_{\mc H} \circ h_{\mc T}: \mc T \to \mc A(\mc T)_{\mc H}$, called the \emph{homological residue field of $\mc T$ at $\mc H$}. The corresponding pure-injective object of \Cref{making the residue objects} is called the \emph{residue object of $\mc T$ at $\mc H$} and is denoted by $\kappa(\mc H)$. 

    By construction, the category $\mc A(\mc T)_{\mc H}^{\fp}$ has only the trivial Serre $\otimes$-ideals, $0$ and $\mc A(\mc T)_{\mc H}^{\fp}$ itself. We highlight a useful consequence of this: if $\mc B \neq 0$ is another locally coherent Grothendieck abelian category with colimit-preserving symmetric monoidal structure and $G\colon \mc A(\mc T)_{\mc H} \to \mc B$ is an exact, colimit-preserving, symmetric monoidal functor that sends finitely presented objects to finitely presented objects, then \cite[Proposition A.6]{BKSruminations} implies that 
    \begin{equation*}
        \Ker(G) = \Loc\left(\Ker(G) \cap \mc A(\mc T)_{\mc H}^{\fp} \right).
    \end{equation*}
    The subcategory $\Ker(G) \cap \mc A(\mc T)_{\mc H}^{\fp}$ is a Serre $\otimes$-ideal of $\mc A(\mc T)_{\mc H}^{\fp}$. So $G \neq 0$ ($G$ is monoidal) implies that $\Ker(G) \cap \mc A(\mc T)_{\mc H}^{\fp} =0$ and hence also $\Ker(G) =0$; this means that $G$ is conservative, not just on the finitely presented part $\mc A(\mc T)_{\mc H}^{\fp}$ but on the whole category~$\mc A(\mc T)_{\mc H}$. 
\end{recollection}

\begin{recollection}\label{homological spectrum vs Balmer spectrum}
    For each homological prime $\mc H \in \Spc^{\mathrm{h}}(\mc T^c)$, its preimage in $\mc T^c$ under the Yoneda embedding 
    \begin{equation*}
       \pi_{\mc T}(\mc H) \vcentcolon= h_{\mc T}^{-1}(\mc H) \cap \mc T^c = \{ x \in \mc T^c \mid \hat{x} \in \mc H\}
    \end{equation*}
    is a prime thick $\otimes$-ideal of $\mc T^c$, i.e.\@ a point of the Balmer spectrum $\Spc(\mc T^c)$, by \cite[Proposition 3.3]{Balmernilpotence}. The map ${\pi_{\mc T} \colon \Spc^{\mathrm{h}}(\mc T^c) \to \Spc(\mc T^c)}$ is a surjection by \cite[Corollary 3.9]{Balmernilpotence}. We say that a homological prime $\mc H \in \Spc^{\mathrm{h}}(\mc T^c)$ \emph{lies over the point $\mc P \in \Spc(\mc T^c)$} if $\pi_{\mc T}(\mc H) = \mc P$. When $\pi_{\mc T}$ is a bijection, we say that $\mc T$ satisfies the \emph{steel condition}. 
\end{recollection}

\begin{remark}
    The steel condition is known to hold for many of the standard tt-categories across algebraic geometry, modular representation theory and stable homotopy theory; e.g.\@ the derived category $\D_{\mathrm{Qcoh}}(X)$ for any quasi-compact quasi-separated scheme $X$, the stable module category $\Stab(kG)$ (and more generally $\DPerm(G;k)$) for a finite group~$G$, the stable homotopy category of topological spectra $\SH$ and its $G$-equivariant version~$\SH(G)$ for a compact Lie group $G$. However, in \cite{nervesofsteelcounterex}, it was shown that the steel condition is not automatic for `big' tt-categories $\mc T$, that is the map $\pi_{\mc T}$ can fail to be injective. 
\end{remark}

\begin{recollection}\label{functoriality_of_the_homological_spectrum}
    Consider a geometric tt-functor $F\colon \mc T \to \mc S$ between `big' tt-categories with right adjoint~$U$. By \cite[Corollary 2.4]{Krause}, both $F$ and $U$ induce exact colimit-preserving functors $\hat{F} \colon \mc A(\mc T) \to \mc A(\mc S)$ and $\hat{U} \colon \mc A(\mc S) \to \mc A(\mc T)$ making the following diagram commutative
\[\begin{tikzcd}
	{\mc T} & {\mc A(\mc T)} \\
	{\mc S} & {\mc A(\mc S) \rlap{ .}}
	\arrow["{h_{\mc T}}", from=1-1, to=1-2]
	\arrow["F"', shift right=2, from=1-1, to=2-1]
	\arrow["{\hat{F}}"', shift right=2, from=1-2, to=2-2]
	\arrow["U"', shift right=2, from=2-1, to=1-1]
	\arrow["{h_{\mc S}}"', from=2-1, to=2-2]
	\arrow["{\hat{U}}"', shift right=2, from=2-2, to=1-2]
\end{tikzcd}\]
    The functor $\hat{F}$ is symmetric monoidal and preserves finitely presented objects. Furthermore, the adjunction $F \dashv U$ yields an adjunction $\hat{F} \dashv \hat{U}$ which also satisfies a projection formula $\hat{U}(M \otimes \hat{F}(N)) \cong \hat{U}(M) \otimes N$ for $M \in \mc A(S)$ and $N \in \mc A(\mc T)$. 

     Now, take a homological prime~$\mc G \in \Spc^{\mathrm{h}}(\mc S^c)$. Then $$ \varphi^h_{F}(\mc G) \vcentcolon= \hat{F}^{-1}(\mc G) \cap \mc A(\mc T)^{\fp} =\{ M \in \mc A(\mc T)^{\fp} \mid \hat{F}(M) \in \mc G\}$$ is a maximal Serre $\otimes$-ideal of $\mc A(\mc T)^{\fp}$, i.e.\@ $\varphi^h_{F}(\mc G)$ defines a homological prime of~$\mc T$. The resulting map $\varphi^h_{F} \colon \Spc(\mc S^c) \to \Spc(\mc T^c)$ makes the following diagram commute 
\[\begin{tikzcd}
	{\Spc^{\mathrm{h}}(\mc S^c)} & {\Spc^{\mathrm{h}}(\mc T^c)} \\
	{\Spc(\mc S^c)} & {\Spc(\mc T^c) \rlap{ .}}
	\arrow["{\varphi_F^h}", from=1-1, to=1-2]
	\arrow["{\pi_{\mc S}}"', two heads, from=1-1, to=2-1]
	\arrow["{\pi_{\mc T}}", two heads, from=1-2, to=2-2]
	\arrow["{\varphi_F}"', from=2-1, to=2-2]
\end{tikzcd}\]
Moreover, the residue object $\kappa(\mc G)$ is a direct summand of $U(\kappa(\mc H))$, where $ \mc H = \varphi^h_{F}(\mc G)$; see \cite[Lemma 5.6 and Theorem 5.10]{bigsupport}. 
\end{recollection}

\begin{recollection}\label{homological_spectrum_of_ttquotients}
A special case of \Cref{functoriality_of_the_homological_spectrum} happens when $F$ is a Verdier localisation $L: \mc T \to \mc T/\Loc(\mc J)$, with $\mc J \subseteq \mc T^c$ a thick $\otimes$-ideal \cite[Proposition 2.9 \& Corollary 3.6]{Balmernilpotence}: the functor $\hat{L}$ is the Gabriel quotient of $\mc A(\mc T)$ by $\Loc(\langle h_{\mc T}(\mc J) \rangle)$ and, for $\varphi^h_L(\mc G) = \mc H$, the homological residue fields $\mc A(\mc T)_{\mc H}$ and $\mc A(\mc T/\Loc(\mc J))_{\mc G}$ are canonically equivalent in such a way that the following diagram commutes 
\[\begin{tikzcd}
	{\mc T} & {\mc A(\mc T)} & {\mc A(\mc T)_{\mc H}} \\
	{\mc T/\Loc(\mc J)} & {\mc A(\mc T/\Loc(\mc J))} & {\mc A(\mc T/\Loc(\mc J))_{\mc G} \rlap{ .}}
	\arrow["h", from=1-1, to=1-2]
	\arrow["L"', two heads, from=1-1, to=2-1]
	\arrow["{Q_{\mc H}}", two heads, from=1-2, to=1-3]
	\arrow["{\hat{L}}", two heads, from=1-2, to=2-2]
	\arrow["{\cong }", from=1-3, to=2-3]
	\arrow["h"', from=2-1, to=2-2]
	\arrow["{Q_{\mc G}}"', two heads, from=2-2, to=2-3]
\end{tikzcd}\]
It follows from the above diagram that 
\[\kappa(\mc G )= L(\kappa(\mc H)) \in \mc T/\Loc(\mc J).\]
Moreover, the map $\varphi^h_L$ yields a bijection between $\Spc^{\mathrm{h}}\left((\mc T/\Loc(\mc J))^c\right)$ and the subset ${\{ \mc H \in \Spc^{\mathrm{h}}(\mc T^c) \mid \pi_{\mc T}(\mc H) \supseteq \mc J\}}$ of $\Spc^{\mathrm{h}}(\mc T^c)$. 
\end{recollection}

The residue objects $\kappa(\mc H)$, for $\mc H \in \Spc^{\mathrm{h}}(\mc T^c)$ (\Cref{homological spectrum}), form the basis for the definition of residual regularity in \Cref{res reg sec}, so we collect some additional results regarding these.

\begin{proposition}[\cite{Balmernilpotence, BKSruminations}]\label{tensor to zero}
    Let $\mc T$ be a `big' tt-category and let $\mc H \in \Spc^{\mathrm{h}}(\mc T^c)$ be a homological prime. We have 
    \begin{enumerate}
        \item The morphism $\kappa(\mc H) \otimes \eta_{\mc H} \colon \kappa(\mc H) \to \kappa(\mc H) \otimes \kappa(\mc H)$ is a split monomorphism.
        \item We can recover $\Loc(\mc H)$ from $\kappa(\mc H)$ as 
        $\Loc(\mc H) = \Ker(\widehat{\kappa(\mc H)} \otimes - )$. 
        \item If $\mc G \in \Spc^{\mathrm{h}}(\mc T^c)$ is distinct from $\mc H$, then ${\kappa(\mc H) \otimes \kappa(\mc G) = 0}$.
    \end{enumerate}
\end{proposition}

To the author's knowledge, the following two lemmas are not noted anywhere, but they are no doubt known to experts as they are straightforward consequences of the theory of homological residue fields. We include the proofs here for posterity. 
\begin{lemma}\label{residue objects are local}
    Let $\mc T$ be a `big' tt-category, $\mc P \in \Spc(\mc T^c)$ and $\mc H \in \pi_{\mc T}^{-1}(\{\mc P\})$. Then the residue object $\kappa(\mc H)$ is $\mc P$-local, i.e.\@ $\kappa(\mc H) \in \mc P^{\perp} = \{ Y \in \mc T \mid \Hom_{\mc T}(a, Y) = 0 \text{ for all } a \in \mc P\}$. 
\end{lemma}
\begin{proof}
    Since $ \overbar{\kappa(\mc H)}$ is injective in the homological residue field $\mc A(\mc T)_{\mc H}$, we know, by isomorphism \eqref{injective iso}, that
    \[\Hom_{\mc T}(X, \kappa(\mc H)) \cong \Hom_{\mc A(\mc T)_{\mc H}}\left(\bar{X}, \overbar{ \kappa(\mc H)}\right),\]
    for any $X \in \mc T$. Now, by definition $\mc P = \pi_{\mc T}(\mc H) = h_{\mc T}^{-1}(\mc H) \cap \mc T^c$. Therefore, $h_{\mc T}(\mc P) \subseteq \mc H$ and thus $\mc P$ is annihilated by~$Q_{\mc H} \circ h: \mc T \to \mc A(\mc T)_{\mc H}$. As such, we find that 
    \[\Hom_{\mc T}(\mc P, \kappa(\mc H)) \cong \Hom_{\mc A(\mc T)_{\mc H}}(0 , \overbar{ \kappa(\mc H)}) =0,\]
    as desired.
\end{proof}

\begin{recollection}\label{local cat}
    Recall that a `big' tt-category $\mc T$ is called \emph{local} if the spectrum $\Spc(\mc T^c)$ has a unique closed point, in which case this closed point must be the zero ideal $\mc M = (0)$. 
    To every point $\mc P \in \Spc(\mc T^c)$, one can associate a local tt-category $\mc T_{\mc P} \vcentcolon= \mc T/\Loc(\mc P)$, called the \emph{stalk category of $\mc T$ at $\mc P$}.  
\end{recollection}

\begin{lemma}
    Let $\mc T$ be a `big' tt-category, $\mc P \in \Spc(\mc T^c)$ and $\mc H \in \pi^{-1}_{\mc T}(\{\mc P\})$. For every point $\mc Q\in \Spc(\mc T^c) \setminus \overline{\{\mc P\}}$ we have $\kappa(\mc H) =0$ in $\mc T_{\mc Q}$. 
\end{lemma}
\begin{proof}
     Fix a point $\mc Q\in \Spc(\mc T^c) \setminus \overline{\{\mc P\}}$. 
     By \Cref{homological_spectrum_of_ttquotients}, all homological primes of $\mc T_{\mc Q}$ `come from' $\Spc^{\mathrm{h}}(\mc T^c)$:
    \[\Spc^{\mathrm{h}}((\mc T_{\mc Q})^c) \cong \{ \mc G \in \Spc^{\mathrm{h}}(\mc T^c) \mid \pi_{\mc T}(\mc G) \supseteq \mc Q\} .\]
    Since $\mc Q \nsubseteq P$, this means that $\mc H \neq \mc G$ for any $ \mc G \in \Spc^{\mathrm{h}}(\mc T^c) $ with $\pi_{\mc T}(\mc G) \supseteq \mc Q$. Hence, by \Cref{tensor to zero}, we have
    \begin{equation*}
      \widehat{\kappa(\mc G)} \otimes \widehat{\kappa(\mc H)} =0 \in \mc A(\mc T) . 
    \end{equation*}
    It follows that $Q_{\mc G} \left(\widehat{\kappa(\mc H)}\right) =0$ for every $\mc G \in \{ \mc G' \in \Spc^{\mathrm{h}}(\mc T^c) \mid \pi_{\mc T}(\mc G') \supseteq \mc Q\} $.
    Applying the nilpotence theorem \cite[Corollary 4.5]{Balmernilpotence} to the image of $\eta_{\mc H}: \unit \to \kappa(\mc H)$ in the tt-category $\mc T_{\mc Q}$, we find that $\eta_{\mc H}^{\otimes n} = 0 \in \mc T_{\mc Q}$ for some $n \geq 1$. However, by \Cref{tensor to zero}, $\kappa(\mc H) \otimes \eta_{\mc H}$ is a split monomorphism, hence $\kappa(\mc H) \otimes \eta_{\mc H}^{\otimes n}$ is a composite of split monomorphisms. Thus $\eta_{\mc H}^{\otimes n} = 0 \in \mc T_{\mc P}$ implies that $\kappa(\mc H) = 0$ in $\mc T_{\mc Q}$. 
\end{proof}

\section{Residual regularity}\label{res reg sec}

The following definition was proposed by Paul Balmer in private communication. 
\begin{definition}
    A `big' tt-category $\mc T$ is \emph{residually regular at the point ${\mc P \in \Spc(\mc T^c)}$} if there exists a homological prime $\mc H \in \pi_{\mc T}^{-1}(\{\mc P\})$ such that the residue object $\kappa(\mc H)$ is a compact object in the stalk category~$\mc T_{\mc P}= \mc T/\Loc(\mc P)$. 
    
    A `big' tt-category $\mc T$ is \emph{residually regular} if it is residually regular at all points of~$\Spc(\mc T^c)$. 
\end{definition}

Any reasonable definition of regularity in tt-geometry should recapture the classical concept of regularity in algebraic geometry, when applied to the derived category of a Noetherian scheme. The following example shows that this is indeed the case for residual regularity, while simultaneously further justifying the terminology for the residue objects.
\begin{example}\label{noetherian rings}
    Let $R$ be a commutative ring with derived category $\D(R)$. Then ${\mc T = \D(R)}$ is a `big' tt-category with $\mc T^c \cong \D^{\perf}(R) = \K^b(\mathrm{proj}(R))$. As a consequence of Thomason's work \cite{Thomasonclassific}, its Balmer spectrum is homeomorphic to the classical Zariski spectrum $\Spec(R)$; see \cite[Corollary 5.6]{BalmerOG}. Moreover, if $\mc P = \mc P(\mf p) \in \Spc(\mc T^c)$ is the point corresponding to a prime ideal $\mf p \in \Spec(R)$, then the stalk category $\mc T_{\mc P}$ is naturally equivalent to the derived category~$\D(R_{\mf p})$. By \cite[Corollary 5.11]{Balmernilpotence}, $\mc T= \D(R)$ satisfies the steel condition. Take $\mc H$ the homological prime corresponding to the Zariski prime~$\mf p \in \Spec(R) \cong \Spc(\mc T^c)$. In \cite[Corollary 3.3]{BalmerCameron}, Balmer and Cameron show that the residue object $\kappa(\mc H)$ is isomorphic to the (classical) residue field $\kappa(\mf p) = R_{\mf p}/\mf p R_{\mf p} \in \D(R)$.

    Altogether, if $R$ is a Noetherian ring, it follows immediately that $\D(R)$ is residually regular at $\mc P(\mf p)$ if and only if $R$ is regular at $\mf p$. More generally, a Noetherian scheme $X$ is regular at a point $x \in X$ if and only if the `big' tt-category $\mc T= \D_{\mathrm{Qcoh}}(X)$ is residually regular at the corresponding point $\mc P(x) \in \Spc(\mc T^c) \cong X$.
\end{example}

\begin{remark}\label{noetherian caveat}
    There is a caveat in order. In commutative algebra (and more generally algebraic geometry), regularity is a condition on Noetherian rings --- this is in part what makes it a well-behaved notion. In tt-geometry, there is not yet any consensus on what a `Noetherian tt-category' is. 
    Just as with regularity, any reasonable definition of a `Noetherian tt-category' should recapture the derived categories of Noetherian rings. 
    At minimum, whatever a `Noetherian tt-category' is, it should have a (topologically) Noetherian Balmer spectrum. This low bar already trips up the stable homotopy category of finite CW-complexes and its $p$-local counterparts, since their Balmer spectra are not Noetherian as observed in \cite[Example 7.11]{Balmer-Favi}; see also \Cref{stable homotopy cat} below.
    
    The expectation is that, for residual regularity to be a well-behaved notion in tt-geometry, `Noetherian' conditions are needed. It may turn out that (some of) the conditions in the statements of this article will be immediate under the `correct' definition of Noetherianity. 
\end{remark}

\begin{proposition}\label{residual regularity implies nerves of steel}
    If a `big' tt-category $\mc T$ is residually regular at $\mc P \in \Spc(\mc T^c)$, then $\pi_{\mc T}^{-1}(\{\mc P\})$ is a singleton. Consequently, if $\mc T$ is residually regular then $\mc T$ satisfies the steel condition.
\end{proposition}
\begin{proof}
    It suffices to prove this for $\mc T$ a local tt-category and $\mc P =(0)$ the unique closed point. After all, by \Cref{functoriality_of_the_homological_spectrum,homological_spectrum_of_ttquotients} the subset $\pi_{\mc T}^{-1}(\{ \mc P \})$ is in bijection with $\pi_{\mc T_{\mc P}}^{-1}(\{(0) \})$. 
    
    So, suppose that $\mc T$ is a local tt-category and $\mc P =(0)$. Take a homological prime $\mc H \in \pi_{\mc T}^{-1}(\{(0)\})$ such that the residue object $\kappa(\mc H)$ is a compact object in $\mc T_{\mc P} = \mc T$. Suppose, for the sake of contradiction, that there exists a homological prime $\mc G \in \pi_{\mc T}^{-1}(\{(0)\})$ distinct from $\mc H$. 
    Since $\pi_{\mc T}(\mc G) =(0)$ by assumption, it follows from \Cref{tensor to zero} that 
    \[ \widehat{\kappa(\mc H)} \in \Ker\left(\widehat{\kappa(\mc G)} \otimes -\right) \cap h_{\mc T}(\mc T^c) =\mc G \cap h_{\mc T}(\mc T^c) = (0), \]
    contradicting $\widehat{\kappa(\mc H)} \neq 0 \in \mc A(\mc T)$. 
\end{proof}

\begin{notation}\label{kappanot}
   Let $\mc P \in \Spc(\mc T^c)$. When $\pi_{\mc T}^{-1}(\{ \mc P\})$ is a singleton (for example, when $\mc T$ is residually regular at $\mc P$), we will sometimes write $\kappa(\mc P)$ for the residue object $\kappa(\mc H)$ with $\mc H$ the \emph{unique} homological prime lying over $\mc P$, and we will accordingly call $\kappa(\mc P)$ the residue object of $\mc T$ at $\mc P$. 
\end{notation}

\begin{remark}\label{inthelit}
    As explained in the introduction, the definition of residual regularity takes inspiration from the characterisation of regularity in commutative algebra involving the residue fields. One could also try to mimick another characterisation of classical regularity: a Noetherian ring $R$ is regular if and only if the \emph{singularity category} $\D^b(\fgmod(R))/ \D^{\perf}(R)$ vanishes. Since $\D^{\perf}(R)$ is exactly the compact part of $\D(R)$, the work lies in finding the appropriate analogue of the bounded derived category $\D^b(\fgmod(R))$. 
    This is like the search for the holy grail, as evidenced by the numerous articles defining such an analogue such as e.g.\@ \cite{singcats, Krausecompletion, kostas, Neeman2026, Rouquier} for triangulated categories and \cite{BarwickLawson, GregGreenlees} for differential graded algebras and structured ring spectra.  Whilst all these approaches have their merits, they all also have drawbacks for use in tt-geometry, either in the sense that they only apply to a small class of tt-categories, or that certain basic compatibilities with the tensor product are not satisfied (whereby, for example, the resulting notion is not local on the Balmer spectrum). 

    Yet another approach to regularity has to do with \emph{generation}. Recall that a Noetherian ring $R$ is regular of finite Krull dimension if and only if $\D^{\perf}(R)$ has a strong generator. Because of this, in the world of triangulated categories, strong generation of the subcategory of compact objects has traditionally been viewed as some kind of homological regularity. In the tt-setting this can suitably be adapted (and thus made into a local property) by requiring not strong generation globally, but strong generation of each stalk category~$(\mc T_{\mc P})^c$ for $\mc P \in \Spc(\mc T^c)$. This formulation relates to residual regularity in the following way: If the residue object $\kappa(\mc P)$ is `cohomologically locally finitely presented' in $\mc T_{\mc P}$ in the sense of \cite{Rouquier}, then strong generation of $(\mc T_{\mc P})^c$ implies immediately that $\kappa(\mc P)$ must be compact in $\mc T_{\mc P}$, by \cite[Corollary 4.29]{Rouquier}. So far, however, we do not know of any sufficiently general conditions to ensure that the residue objects are cohomologically locally finitely presented. In a similar fashion, in \cite{BIKP}, regularity is defined for certain tt-categories as the strong generation of particular `local' subcategories for each point of the Balmer spectrum. 

    Investigating the connections between all these different notions of regularity is an interesting topic for further research.
\end{remark}

\begin{example}\label{tt fields are regular}
    In \cite{BKSruminations}, Balmer--Krause--Stevenson define the concept of a \emph{tensor triangular field} (tt-field): A nontrivial `big' tt-category $\mc F$ is called a tt-field if every object $X \in \mc F$ is a coproduct of compact objects, and if every nonzero $X$ is $\otimes$-faithful, meaning that the functor $X \otimes -: \mc F \to \mc F$ is faithful. Examples include the derived category of a field $k$ in commutative algebra, modules over Morava $K$-theory in topology, and the stable module category of the cyclic group of order $p$ in modular representation theory. 
    
    In op.\@ cit.\@, it is shown that both the Balmer spectrum and the homological spectrum of a tt-field $\mc F$ consist of only the trivial proper ideal: $\Spc(\mc F^c) =\{(0)\}$ and $\Spc^{\mathrm{h}}(\mc F^c) = \{(0)\}$. Moreover, it is shown that $\hat{X}$ is injective in $\mc A(\mc F)$ for all $X \in \mc F$. Thus, any tt-field is residually regular and the residue object at the unique point is simply the unit $\unit$. 
\end{example}

\begin{corollary}\label{going to tt-field}
    Let $F\colon \mc T \to \mc F$ be a geometric tt-functor to a tt-field $\mc F$, with right adjoint~$U$. If $U(\unit)$ is compact in $\mc T$, then $\mc T$ is residually regular at $\mc P\vcentcolon = \varphi_F\big((0)\big)$. 
\end{corollary}
\begin{proof}
    Set $\mc H \vcentcolon= \varphi^h_F((0)) \in \Spc^{\mathrm{h}}(\mc T^c)$. The residue object $\kappa(\mc H)$ is a direct summand of $U(\unit)$ in $\mc T$, see \Cref{functoriality_of_the_homological_spectrum}. Therefore, if $U(\unit)$ is compact, so is $\kappa(\mc H)$. 
\end{proof}

We recall the following definition of \cite{BHS}:
\begin{definition}
    Let $\mc K$ be an essentially small tt-category. A point $\mc P \in \Spc(\mc K)$ is called weakly visible if the singleton $\{ \mc P\}$ is the intersection of a Thomason subset and the complement of a Thomason subset. The spectrum $\Spc(\mc K)$ is called weakly Noetherian if every point is weakly visible. 
\end{definition}

\begin{remark}\label{visibiliy noeth}
    The above terminology was inspired by the following fact: in \cite{Balmer-Favi} a point $\mc P \in \Spc(\mc K)$ is called \emph{visible} if the closure $\overline{\{ \mc P\}}$ is a Thomason subset. In this case, $\mc P$ is certainly also weakly visible, as we have 
    \begin{equation*}
        \{ \mc P\} = \overline{\{ \mc P\}} \cap \gen(\mc P),
    \end{equation*}
    where $\gen(\mc P)$ is the subset consisting of all generalisations of $\mc P$ (its complement is the support of the ideal $\mc P$, which is always a Thomason subset). Furthermore, in \cite{Balmer-Favi} it is noted that $\Spc(\mc K)$ is Noetherian if and only if every point is visible. 
\end{remark}

\begin{lemma}\label{support is point}
    Let $\mc T$ be a local `big' tt-category $\mc T$. If $\mc T$ is residually regular at its unique closed point $\mc M$, then $\supp_{\mc T}(\kappa(\mc M)) = \{ \mc M\}$. 
\end{lemma}
\begin{proof}
    If $\mc T$ is residually regular at the point $\mc M = (0)$, then $\kappa(\mc M) \neq 0$ is compact in~$\mc T_{\mc M} = \mc T$. Therefore, we certainly have $\mc M \in \supp_{\mc T}(\kappa(\mc M))$. Now take $\mc P\in \Spc(\mc T^c)$ distinct from $\mc M$. For any $\mc G \in \pi_{\mc T}^{-1}(\{\mc P\})$, we have by \Cref{tensor to zero}
    \[ \kappa(\mc M) \in h_{\mc T}^{-1} \left(\Ker\left(\widehat{\kappa(\mc G)} \otimes -\right)  \right) \cap \mc T^c  =\mc P. \]
    This implies that $\mc P \notin \supp_{\mc T}(\kappa(\mc M))$. 
\end{proof}

\begin{corollary}\label{visibility of unique closed point}
    If a `big' tt-category $\mc T$ is residually regular at a point $\mc P \in \Spc(\mc T^c)$, then $\mc P$ is weakly visible in $\Spc(\mc T^c)$. Consequently, if $\mc T$ is residually regular then $\Spc(\mc T^c)$ is weakly Noetherian.
\end{corollary}
\begin{proof}
    By \cite[Remark 2.9]{BHS}, $\mc P$ is weakly visible in $\Spc(\mc T^c)$ if and only if the unique closed point $(0)$ of the stalk category $\mc T_{\mc P}$ is visible (see \Cref{visibiliy noeth}). If $\mc T$ is residually at~$\mc P$, then it follows by \Cref{support is point} that
    \begin{equation*}
       \overline{\{(0)\}} =  \left\{ (0) \right\} =\supp_{\mc T_{\mc P}} (\kappa(\mc P)),
    \end{equation*}
    which exhibits $\overline{\{(0)\}} $ as a Thomason subset of $\Spc((\mc T_{\mc P})^c)$. 
\end{proof}

\begin{example}\label{stable homotopy cat}
    For any prime number $p$, the $p$-local stable homotopy category of topological spectra $\SH_{(p)}$ is a local tt-category. Its unique closed point is not weakly visible, as observed in \cite[Example 7.11]{Balmer-Favi} and therefore $\SH_{(p)}$ is not residually regular at this point. By extension, the stable homotopy category of spectra $\SH$ is not residually regular at any of its closed points (since residual regularity is a stalk-local property). 
\end{example}

\begin{remark}
    It is certainly not true that residual regularity of the tt-category $\mc T$ implies that the spectrum $\Spc(\mc T^c)$ is Noetherian: for any non-Noetherian absolutely flat commutative ring $R$ (e.g.\@ an infinite product of fields), the derived category $\D(R)$ is residually regular but $\Spc(\D^{\perf}(R)) \cong \Spec(R)$ is not Noetherian; see~\cite[Section 3]{Stevensonabsflatrings}. This also goes to show that residual regularity should likely only be considered with additional Noetherian hypotheses (cf.\@ \Cref{noetherian caveat}).
\end{remark}

\begin{recollection}\label{constructible top}
    Let $\mc K$ be an essentially small tt-category. A subset of $\Spc(\mc K)$ is called \emph{constructible} if it can be obtained by finite union and finite intersection from the quasi-compact opens of $\Spc(\mc K)$ and their complements (equivalently, from supports of objects and their complements). A \emph{proconstructible} subset of $\Spc(\mc K)$ is an arbitrary intersection of constructibles. 
    We refer the reader to \cite[Section 1.3]{spectralbook} for more on the constructible topology of spectral spaces. 
\end{recollection}

\begin{proposition}\label{residue regularity at closed points}
    Let $\mc T$ be a `big' tt-category and let $\mc M$ be a closed visible point of $\Spc(\mc T^c)$. If $\mc T$ is residually regular at $\mc M$, then $\kappa(\mc M)$ is a compact object in~$\mc T$. 
\end{proposition}
\begin{proof}
    Since $\kappa(\mc M)$ is compact in $\mc T_{\mc M}$ by assumption, we know that there exists an object $a \in \mc T^c$ and an isomorphism in $\mc T_{\mc M}$:
    \[\kappa(\mc M) \oplus \Sigma \kappa(\mc M) \cong a \in \mc T_{\mc M}.\]
    Of course, then we have 
    \begin{equation}\label{supports}
        \supp_{\mc T_{\mc M}}(a) = \supp_{\mc T_{\mc M}}(\kappa(\mc M)).
    \end{equation}
    Identifying $\Spc((\mc T_{\mc M})^c)$ with the subset $\gen(\mc M) = \{ \mc P \in \Spc(\mc T^c) \mid \mc M \subseteq \mc P\}$ of $\Spc(\mc T^c)$, equality \eqref{supports} and \Cref{support is point} imply that 
    \[\gen(\mc M) \cap \supp_{\mc T}(a)  = \{ \mc M\}.\]
    Now set $Y = \supp_{\mc T}(a) \setminus \{ \mc M\} $. We claim that $Y$ is a closed Thomason subset of $\Spc(\mc T^c)$. To see this, note that $Y = \supp_{\mc T}(a) \cap \{ \mc M\}^c$ is a constructible subset since $\mc M$ is visible by assumption. Furthermore, $Y$ is certainly specialisation closed because $\supp_{\mc T}(a)$ is and $\gen(\mc M) \cap Y  = \varnothing$. Any specialisation closed constructible set is in fact closed, see \cite[Theorem 1.5.4(i)]{spectralbook}. Moreover, the complement $Y^c$ is quasi-compact since it is the union of a point and a quasi-compact subset: $Y^c = \{ \mc M\} \cup \supp_{\mc T}(a)^c $. Thus, $Y$ is a closed Thomason subset.

    So, we have decomposed $\supp_{\mc T}(a)$ into two disjoint closed Thomason subsets: 
    \[ \supp_{\mc T}(a)= Y \cup \{ \mc M\}.\]
    By the decomposition theorem \cite[Theorem 2.11]{supportsandfiltrations}, $a$ can itself by decomposed as a direct sum $a \cong b \oplus c$ for some objects $b, c \in \mc T^c$ with $\supp_{\mc T}(b) = Y$ and $\supp_{\mc T}(c) = \{\mc M \}$. We note that $c$ is an $\mc M$-local object since $\gen(\mc M)^c \cap \supp_{\mc T}(c) = \varnothing$. Furthermore, as $\mc M \notin Y =\supp_{\mc T}(b)$, we have $b \in \mc M$. It follows that $b =0 \in \mc T_{\mc M}$ and thus, in $\mc T_{\mc M}$ we have
    \[\kappa(\mc M) \oplus \Sigma \kappa(\mc M) \cong  a \cong c \in \mc T_{\mc M}.\]
    To summarise, the two $\mc M$-local objects $\kappa(\mc M) \oplus \Sigma \kappa(\mc M)$ (see \Cref{residue objects are local}) and $c$ become isomorphic in $\mc T_{\mc M}$. As such, they must already be isomorphic in $\mc T$. Therefore, the residue object $\kappa(\mc M)$ is compact in $\mc T$ as a direct summand of the compact object $c$. 
\end{proof}
Using the following lemma, we can remove the visibility assumption in the above proposition at the cost of requiring the spectrum to be \emph{semilocal}, i.e.\@ requiring $\Spc(\mc T^c)$ to have only finitely many closed points.  
\begin{lemma}\label{weak vis}
    Let $\mc K$ be an essentially small tt-category and let $\mc M$ be a closed point of~$\Spc(\mc K)$. Suppose that the spectrum $\Spc(\mc K)$ is semilocal. If $\mc M$ is weakly visible, then $\mc M$ is visible. 
\end{lemma}
\begin{proof}
    By \cite[Remark 2.8]{BHS}, the weak visibility of $\mc M$ implies the existence of a closed Thomason subset $Z \subseteq \Spc(\mc K)$ such that 
    \begin{equation*}
        Z \cap \gen(\mc M) = \{\mc M \}. 
    \end{equation*}
    We claim that $Y = Z \setminus \{ \mc M\}$ is a closed Thomason subset. For this, let $\mc M_1, \mc M_2, \dots, \mc M_k$ denote the finitely many closed points of $\Spc(\mc K)$ and, say, $\mc M_1 = \mc M$. Then 
    \begin{equation*}
        \Spc(\mc K) = \bigcup_{i=1}^k \gen(\mc M_i).
    \end{equation*}
    So, using that $Y \cap \gen(\mc M_1) = \varnothing$, we find 
    \begin{equation*}
        Y = \bigcup_{i=1}^k Y \cap \gen(\mc M_i) = \bigcup_{i=2}^k Y \cap \gen(\mc M_i) = \bigcup_{i=2}^k Z \cap \gen(\mc M_i) = Z \cap  \bigcup_{i=2}^k \gen(\mc M_i).
    \end{equation*}
   Since each $\gen(\mc M_i)$ is the complement of a Thomason subset, this shows that $Y$ is a proconstructible subset (see \Cref{constructible top}). We also have that $Y$ is specialisation closed since $Z$ is and $\gen(\mc M) \cap Y  = \varnothing$. 
    Using \cite[Theorem 1.5.4(i)]{spectralbook}, it follows that $Y$ is in fact a closed subset. Moreover, the complement $Y^c$ is quasi-compact as the union of a point and a quasi-compact subset: $Y^c = \{ \mc M\} \cup Z^c $. So $Y$ is indeed a closed Thomason subset. This implies that $\{ \mc M \}^c = Y \cup Z^c$ is constructible. Since $\mc M$ is a closed point, $\{ \mc M\}^c$ is certainly an open subset as well. So, by application of \cite[Corollary 1.3.18(i)]{spectralbook}, the subset $\{ \mc M \}^c$ is in fact quasi-compact open. Therefore we conclude that $\{ \mc M\}$ is a closed Thomason subset, i.e.\@ $\mc M$ is a visible point.
\end{proof}

\begin{corollary}
    Let $\mc T$ be a `big' tt-category and let $\mc M$ be a closed point of $\Spc(\mc T^c)$. Suppose that the spectrum $\Spc(\mc T^c)$ is semilocal. If $\mc T$ is residually regular at $\mc M$, then $\kappa(\mc M)$ is a compact object in~$\mc T$. 
\end{corollary}
\begin{proof}
    By \Cref{visibility of unique closed point} and \Cref{weak vis}, the closed point $\mc M$ is visible. The statement now follows by application of \Cref{residue regularity at closed points}. 
\end{proof}

\begin{remark}\label{residual regularity of local tt-categories}
    In algebraic geometry, the locus of regular points of a Noetherian scheme $X$ is a generalisation closed subset. As a consequence, $X$ is regular if and only if it is regular at all its closed points. One can wonder (and hope) whether the same will be true for residual regularity of tt-categories. Further research will tell. 
\end{remark}

The next result can be viewed as a generalisation of \Cref{going to tt-field}. 
\begin{proposition}\label{left to right}
    Let $F \colon \mc T \to \mc S$ be a geometric tt-functor with right adjoint $U$ and let $\mc Q \in \Spc(\mc S^c)$ and $\mc P \vcentcolon= \varphi_{F}(\mc Q)$. Suppose that the following conditions hold:  
    \begin{enumerate}[label=(\arabic*)]
        \item The point $ \mc Q$ is visible in $\Spc(\mc S^c)$.
        \item The point $\mc Q$ is closed in the fibre $\varphi_F^{-1}(\{\mc P\})$.
        \item The right adjoint $U$ preserves compact objects. 
    \end{enumerate}
    If $\mc S$ is residually regular at $\mc Q$, then $\mc T$ is residually regular at $\mc P$.
\end{proposition}
\begin{proof}
    Consider the thick $\otimes$-ideal $\langle F(\mc P) \rangle $ of $\mc S^c$ generated by $F(\mc P)$. Since $\varphi_F(\mc Q) = \mc P$, the functor $F$ induces a geometric tt-functor~$\tilde{F} \colon \mc T_{\mc P} \to \mc S/\Loc(\langle F(\mc P) \rangle)$. Using the projection formula for $F \dashv U$, we note that $U(\mc S^c \otimes F(\mc P)) \subseteq \Loc(\mc P)$, as $\Loc(\mc P)$ is a $\otimes$-ideal in $\mc T$. It follows 
    \begin{equation*}
        U\left(\langle F(\mc P) \rangle\right) = U\left( \thick(\mc S^c \otimes F(\mc P))\right) \subseteq \thick \left( U(\mc S^c \otimes F(\mc P))\right) \subseteq \Loc(\mc P). 
    \end{equation*}
    Therefore, by the universal property, $U$ induces a functor $\tilde{U}\colon \mc S/\Loc(\langle F(\mc P) \rangle) \to \mc T_{\mc P} $ such that $\tilde{F}$ is left adjoint to $\tilde{U}$. 
    Altogether, we get a diagram of adjunctions
\begin{equation}\label{commm diagram}
    \begin{tikzcd}
	{\mc T} & {\mc T_{\mc P}} \\
	{\mc S} & {\mc S/\Loc(\langle F(\mc P) \rangle)}
	\arrow[two heads, from=1-1, to=1-2]
	\arrow[""{name=0, anchor=center, inner sep=0}, "F"', shift right=2, from=1-1, to=2-1]
	\arrow[""{name=1, anchor=center, inner sep=0}, "{\tilde{F}}"', shift right=2, from=1-2, to=2-2]
	\arrow[""{name=2, anchor=center, inner sep=0}, "U"', shift right=2, from=2-1, to=1-1]
	\arrow["L"', two heads, from=2-1, to=2-2]
	\arrow[""{name=3, anchor=center, inner sep=0}, "{\tilde{U}}"', shift right=2, from=2-2, to=1-2]
	\arrow["\dashv"{anchor=center}, draw=none, from=0, to=2]
	\arrow["\dashv"{anchor=center}, draw=none, from=1, to=3]
\end{tikzcd}
\end{equation}
   in which the two squares commute.

Since $\varphi_F(\mc Q) = \mc P$ it follows that $\langle F(\mc P) \rangle \subseteq \mc Q$, whereby $\mc Q$ corresponds to a point $\tilde{\mc Q}$ of $\Spc(\mc S^c/\langle F(\mc P) \rangle)$ that gets mapped to the unique closed $(0) \in \Spc((\mc T_{\mc P})^c)$ by $\varphi_{\tilde{F}}$. We claim that $\tilde{\mc Q}$ is a closed visible point of $\Spc(\mc S^c/\langle F(\mc P) \rangle)$. After all, if there exists a prime ideal $\mc Q'$ of $\mc S^c$ such that 
\begin{equation*}
    F(\mc P)  \subseteq \mc Q' \subseteq \mc Q,
\end{equation*}
then 
\[\mc P \subseteq F^{-1} F(\mc P) \subseteq \varphi_F(\mc Q') \subseteq \varphi_F( \mc Q) = \mc P.\]
Therefore, $\varphi_F(\mc Q') =\mc P$, i.e.\@ $\mc Q'$ is a point in the fibre $\varphi_F^{-1}\{\mc P\}$ with $\mc Q' \subseteq \mc Q$. As $\mc Q$ is closed in the fibre, we must have $\mc Q' = \mc Q$. Thus, $\tilde{\mc Q}$ is a closed point of $\Spc(\mc S^c/\langle F(\mc P) \rangle)$. By a similar reasoning, using that $\mc Q$ is visible in $\Spc(\mc S^c)$ by assumption, it follows that $\tilde{\mc Q}$ is a visible point of the spectrum of $\mc S^c/\langle F(\mc P)\rangle$.

Now, let $\mc G$ be the unique homological prime in $\Spc^{\mathrm{h}}(\mc S^c)$ lying over $\mc Q$ and set $\mc H \vcentcolon= \varphi^h_F(\mc G) \in \Spc^{\mathrm{h}}(\mc T^c)$. By \Cref{functoriality_of_the_homological_spectrum}, the residue object $\kappa(\mc H)$ is a direct summand of $U(\kappa(\mc G))$ in $\mc T$. It follows, by the commutativity of diagram \eqref{commm diagram}, that $\kappa(\mc H)$ is a direct summand of $\tilde{U}(\kappa(\mc G))$ in $\mc T_{\mc P}$. By \Cref{residue regularity at closed points}, the residue object $\kappa(\mc G) \in \mc S$ is compact in $\mc S/\Loc(\langle F(\mc P)\rangle)$. The commutativity of diagram \eqref{commm diagram} implies that $\tilde{U}$ preserves compact objects, since by assumption $U$ does. Therefore, the residue object $\kappa(\mc H)$ is compact in $\mc T_{\mc P}$ as a direct summand of the compact object $\tilde{U}(\kappa(\mc Q)) \in \mc T_{\mc P}$.
\end{proof}

\section{A descent-ascent statement}\label{descent stat sec}
In this section, we will prove a converse to the statement of \Cref{left to right} under stronger conditions. We begin by recalling some of the theory of separable ring objects in symmetric monoidal categories. 

\begin{recollection}
    Let $(\mc C, \otimes, \unit)$ be a symmetric monoidal idempotent complete additive category. A \emph{ring object} $A$ in $\mc C$ is a monoid $(A, \mu \colon A \otimes A \to A, u \colon \unit \to A)$ with associative multiplication $\mu$ and two-sided unit $u$. It is called commutative when $\mu \circ (12) = \mu$, where $(12)$ denotes the isomorphism permuting the tensor factors $X_1 \otimes X_2 \xrightarrow{\sim} X_2 \otimes X_1$. All ring objects in this paper will be commutative and we often will simply call $A$ a \emph{ring in $\mc C$}.
    
    A (left) \emph{$A$-module} is a pair $(X, m \colon A \otimes X \to X)$ with $X \in \mc C$ and $m$ a morphism in $\mc C$ compatible with the ring structure in the usual way. The \emph{Eilenberg-Moore category} $A\text{-}\Mod_{\mc C}$ has left $A$-modules as its objects and left $A$-linear maps as its morphisms. It is additive and idempotent complete.
    Every object $X \in \mc C$ gives rise to a free $A$-module $F_A(X) \vcentcolon= A \otimes X$ with $A$-action given by $\mu \otimes X$. We call the functor $F_A: \mc C \to A\text{-} \Mod_{\mc C}$ the \emph{extension-of-scalars functor}, and write $U_A$ for its forgetful right adjoint: 
\[\begin{tikzcd}
	{\mc C} \\
	{A\text{-}\Mod_{\mc C} \rlap{ .}}
	\arrow[""{name=0, anchor=center, inner sep=0}, "{F_A}"', shift right=2, from=1-1, to=2-1]
	\arrow[""{name=1, anchor=center, inner sep=0}, "{U_A}"', shift right=2, from=2-1, to=1-1]
	\arrow["\dashv"{anchor=center}, draw=none, from=0, to=1]
\end{tikzcd}\]

A ring $A$ in $\mc C$ is called \emph{separable} if the multiplication map $\mu$ admits an $A\text{-}A$-bilinear section $\sigma: A \to A\otimes A$, i.e.\@ we have $\mu \sigma = 1_A$ and the diagram 
\[\begin{tikzcd}
	& {A \otimes A} & \\
	{A \otimes A \otimes A} & A & {A \otimes A \otimes A \rlap{ .}} \\
	& {A \otimes A}
	\arrow["{\sigma \otimes 1}"', from=1-2, to=2-1]
	\arrow["\mu", from=1-2, to=2-2]
	\arrow["{1 \otimes \sigma}", from=1-2, to=2-3]
	\arrow["{1 \otimes \mu}"', from=2-1, to=3-2]
	\arrow["\sigma", from=2-2, to=3-2]
	\arrow["{\mu \otimes 1}", from=2-3, to=3-2]
\end{tikzcd}\]
is commutative. 
When the ring $A$ is separable, the counit of the adjunction $F_A \dashv U_A$ has a natural section, defined for any $A$-module $(X, m)$ by 
\begin{equation}\label{splitting}
    (1_A \otimes m) \circ (\sigma \otimes 1_X) \circ (u \otimes 1_X): X \to A \otimes X.
\end{equation}
It follows that $(X, m)$ is a direct summand of the free module $F_AU_A(X, m) = F_A(X)$.

Moreover, if $ A$ is separable, $A\text{-}\Mod_{\mc C}$ can be endowed with a symmetric monoidal structure under which $F_A$ becomes symmetric monoidal, see \cite[Proposition 1.11]{Bregje}. In this case, the adjunction $F_A \dashv U_A$ satisfies a projection formula: for all $Y \in \mc C$ and $X \in A\text{-}\Mod_{\mc C}$ there is a natural isomorphism $U_A(X \otimes F_A(Y)) \cong U_A(X) \otimes Y$. 

Finally, we note that when $A$ is a rigid, commutative and separable ring object in $\mc C$, the adjunction $F_A \dashv U_A$ is \emph{ambidextrous}, that is, we also have $U_A \dashv F_A$. After all, by \cite[Proposition 2.46]{Sandersetale}, these assumptions imply that $A$ is in fact a Frobenius ring. In that case, the ambidexterity of the adjunction is classical. 
\end{recollection}

The following proposition records some properties of separable rings in abelian categories, tailored to our needs. 
\begin{proposition}\label{sep ext in abelian cats}
    Let $\mc A$ be a locally coherent Grothendieck abelian category with colimit-preserving symmetric monoidal structure and let $A \in \mc A$ be a rigid, commutative and separable ring object. Then we have:
    \begin{enumerate}
        \item The category $A\text{-}\Mod_{\mc A}$ is a locally coherent Grothendieck category that can be endowed with a colimit-preserving symmetric monoidal structure. 
        Consequently, the functors $F_A$ and $U_A$ are exact and preserve colimits, finitely presented objects and injective objects. 
        \item \label{sep ext in abelian cats (b)} For every Serre $\otimes$-ideal $\mc I$ of $\mc A^{\fp}$, the Serre subcategory $U_A^{-1}(\mc I)$ of $\left(A\text{-}\Mod_{\mc A}\right)^{\fp}$ is a $\otimes$-ideal which equals the Serre $\otimes$-ideal $\langle F_A(\mc I)\rangle$ generated by $F_A(\mc I)$.
        \end{enumerate}
        Furthermore, if $\mc B$ is another locally coherent Grothendieck abelian category with colimit-preserving symmetric monoidal structure and $F\colon \mc A \to \mc B$ is an exact, colimit-preserving, symmetric monoidal functor that sends finitely presented objects to finitely presented objects, we have:
        \begin{enumerate}[resume]
        \item The object $B \vcentcolon =F(A) \in \mc B$ is a rigid, commutative and separable ring object and there exists an exact, colimit-preserving, symmetric monoidal functor $\tilde{F}\colon A\text{-}\Mod_{\mc A} \to B\text{-}\Mod_{\mc B}$ which sends finitely presented objects to finitely presented objects and such that $F_B F \cong \tilde{F} F_A$ and $F U_A \cong U_B \tilde{F}$.
        \item If $F$ is the Gabriel quotient $Q_{\mc I} \colon \mc A \to \mc A/\Loc(\mc I)$ where $\mc I \subseteq \mc A^{\fp}$ is a Serre $\otimes$-ideal, then $\tilde{F}$ identifies with the Gabriel quotient of $A\text{-}\Mod_{\mc A}$ by $\Loc(U_A^{-1}(\mc I))$. 
    \end{enumerate}
\end{proposition}
\begin{proof}
    Most of (a) is straightforward using the fact that $U_A \dashv F_A \dashv U_A$, that $U_A$ is faithful (in particular, it creates colimits and reflects exactness) and the splitting \eqref{splitting}. As mentioned before, the fact that $A\text{-}\Mod_{\mc A}$ can be endowed with a symmetric monoidal structure can be found in \cite[Proposition 1.11]{Bregje}. From the description of the tensor product in loc.\@ cit.\@ it follows readily that this tensor product is colimit-preserving in both variables.

    For (b), the projection formula guarantees that $U_A^{-1}(\mc I)$ is a Serre $\otimes$-ideal of $\left(A\text{-}\Mod_{\mc A}\right)^{\fp}$. The asserted equality with $\langle F_A(\mc I) \rangle$ follows from the splitting \eqref{splitting}.

    Now assume the hypotheses of part (c). It is clear that $B = F(A)$ inherits the structure of a commutative, rigid and separable ring, as $F$ is monoidal. The existence of an additive symmetric monoidal functor $\tilde{F}$ with $F_B F \cong \tilde{F} F_A$ and $F U_A \cong U_B \tilde{F}$ is guaranteed by \cite[Proposition 1.2.7]{Pauwelsphd}. On objects $(X, m) \in A\text{-}\Mod_{\mc A}$, the functor $\tilde{F}$ is defined by $\tilde{F}(X, m) = F(X)$ with $B$-module structure given by $$B \otimes F(X) \cong F(A \otimes X) \xrightarrow{F(m)} F(X),$$ and on morphisms $\tilde{F}$ is given by $\tilde{F}(f) = F(f)$. 
    The remaining properties of $\tilde{F}$ (exactness, preservation of colimits and finitely presented objects) all follow from the natural isomorphisms $F_B F \cong \tilde{F} F_A, F U_A \cong U_B \tilde{F}$ and the relevant properties of $F, F_A, F_B, U_A$ and~$U_B$.

    Lastly, if $\mc I$ is a Serre $\otimes$-ideal of $\mc A^{\fp}$ and $F$ is the localisation $Q_{\mc I} \colon \mc A \to \mc B \vcentcolon= \mc A/\Loc(\mc I)$ with right adjoint $R_{\mc I}$, we note that $\tilde{F}=\tilde{Q}_{\mc I}$ has a right adjoint $\tilde{R}_{\mc I}$ given by sending a $Q_{\mc I}(A)$-module $(Y, n)$ to the $A$-module $R_{\mc I}(Y)$ with multiplication
    \begin{equation*}
        A \otimes R(Y) \xrightarrow{\eta \otimes 1} RQ(A) \otimes R(Y) \to R(Q(A) \otimes Y) \xrightarrow{R(n)} R(Y) ,
    \end{equation*}
    where $\eta$ is the unit of the adjunction $Q_{\mc I} \dashv R_{\mc I}$. Here we have used that $R_{\mc I}$ is a lax monoidal functor. Since $Q_{\mc I}R_{\mc I} \cong \id$ ($R_{\mc I}$ is fully faithful) we also have that $\tilde{Q}_{\mc I}\tilde{R}_{\mc I} \cong \id$. Therefore, $\tilde{Q}_{\mc I}$ induces an equivalence $A\text{-}\Mod_{\mc A}/ \Ker(\tilde{Q}_{\mc I}) \xrightarrow{\sim} Q_{\mc I}(A)\text{-}\Mod_{\mc B}$. As
    \begin{equation*}
        \Ker(\tilde{Q}_{\mc I}) = U_A^{-1}(\Loc(\mc I)) = \Loc(U_A^{-1}(\mc I)),
    \end{equation*}
    this proves part (d). 
\end{proof}

\begin{recollection}
    In the tt-setting, Balmer proved in \cite{tt-rings} that extensions along separable ring objects
preserve the triangulation\footnote{This result requires a stronger definition of `triangulated category' than the usual definition à la Verdier. Luckily, any enhanceable tt-category satisfies this stronger axiomatic, so there is no real restriction in terms of applications; see \cite[Section 5]{tt-rings}. }: when $\mc T$ is a `big' tt-category and $A \in \mc T^c$ is a separable ring, $A\text{-}\Mod_{\mc T}$ is a `big' tt-category, extension-of-scalars $F_A$ is a geometric tt-functor and $U_A$ is exact. The exact triangles in $A\text{-}\Mod_{\mc T}$ are exactly those created by~$U_A$. Moreover, since we have $U_A \dashv F_A$, the functor $U_A$ preserves compact objects.

The notion of the \emph{degree of a separable ring} $A$ in a tt-category $\mc T$ has been introduced in \cite{finitedegree} and is denoted by $\deg_{\mc T}(A)$. We will use it as a black box here. Having finite degree is a mild hypothesis which holds for the compact separable rings in all `standard' tt-categories in algebraic geometry, homotopy
theory or modular representation theory. Nevertheless, it is not a redundant condition as shown in \cite{infdegttrings}. 
\end{recollection}

\begin{definition}
    Let $F\colon \mc T \to \mc S$ be a geometric tt-functor between `big' tt-categories. We say that $F$ is a \emph{finite separable extension} if there is a compact and separable ring $A$ in $\mc T$ and an equivalence of tt-categories $\mc S \cong A\text{-}\Mod_{\mc T}$ such that $F \colon \mc T \to \mc S$ is identified with the extension-of-scalars functor $F_A \colon \mc T \to  A\text{-}\Mod_{\mc T}$. If the ring $A$ is of finite degree in $\mc T$, then we say that $F$ is a \emph{finite separable extension of finite degree}. 
\end{definition}

It should come as no surprise that we wish to apply \Cref{sep ext in abelian cats} to the case where the categories $\mc A$ and $\mc B$ are (quotients of) the module categories of some `big' tt-categories. The key to unlock this is given by the following lemma: 
\begin{lemma}\label{compatibilities}
Let $\mc T$ be a `big' tt-category and $A \in \mc T^c$ a separable ring. Then the object $\hat{A} \in \mc A(\mc T)$ is a rigid, commutative and separable ring object and there exists an equivalence ${E \colon \mc A(A\text{-}\Mod_{\mc T}) \to \hat{A}\text{-}\Mod_{\mc A( \mc T)}}$ such that $E \circ \hat{F}_A\cong F_{\hat{A}}$ and $U_{\hat{A}} \circ E \cong \hat{U}_A$:
\[\begin{tikzcd}
	{\mc A(\mc T)} & {\mc A(\mc T)} \\
	{\mc A(A\text{-}\Mod_{\mc T})} & {\hat{A}\text{-}\Mod_{\mc A( \mc T)} \rlap{ .}}
	\arrow[equals, from=1-1, to=1-2]
	\arrow["{\hat{F}_A}"', shift right=3, from=1-1, to=2-1]
	\arrow["{F_{\hat{A}}}"', shift right=3, from=1-2, to=2-2]
	\arrow["{\hat{U}_A}"', shift right=3, from=2-1, to=1-1]
	\arrow["E"', from=2-1, to=2-2]
	\arrow["\cong", draw=none, from=2-1, to=2-2]
	\arrow["{U_{\hat{A}}}"', shift right=3, from=2-2, to=1-2]
\end{tikzcd}\]
Moreover, for every Serre $\otimes$-ideal $\mc I$ of $\mc A(\mc T)^{\fp}$, there is a commutative diagram 
\[\begin{tikzcd}
	{\mc A(\mc T)} & {\frac{\mc A(\mc T)}{{\Loc(\mc I)}}} \\
	{\mc A(A\text{-}\Mod_{\mc T})} \\
	{\frac{\mc A(A\text{-}\Mod_{\mc T})}{\Loc(\langle \hat{F}_A(\mc I)\rangle)}} & {\bar{A}\text{-} \Mod_{\frac{\mc A(\mc T)}{{\Loc(\mc I)}}} \rlap{ ,}}
	\arrow["{Q_{\mc I}}", two heads, from=1-1, to=1-2]
	\arrow["{\hat{F}_A}"', from=1-1, to=2-1]
	\arrow["G", from=1-2, to=3-1]
	\arrow["{F_{\bar{A}}}"', from=1-2, to=3-2]
	\arrow[two heads, from=2-1, to=3-1]
	\arrow["\cong", from=3-1, to=3-2]
\end{tikzcd}\]
where $\bar{A} = Q_{\mc I}(\hat{A})$ and $G$ is induced by the universal property of $Q_{\mc I}$. 
\end{lemma}
\begin{proof}
    The fact that $\hat{A}$ is a commutative and separable ring object is immediate since the restricted Yoneda embeding $h_{\mc T} \colon \mc T \to \mc A(\mc T)$ is a monoidal functor. Moreover, as $A$ is compact, hence rigid, so is $\hat{A}$. The existence of the equivalence $E$ follows by an analogue of Beck Monadicity, see \cite[Theorem 2.9]{restrict}: The adjunction $\hat{F}_A \dashv \hat{U}_A$ satisfies the projection formula and the counit has a natural section that is simply lifted from the section of the counit of ${F}_A \dashv {U}_A$.

    For the second part, by \Cref{sep ext in abelian cats}, the quotient $Q_{\mc I}\colon \mc A(\mc T) \to \frac{\mc A(\mc T)}{\Loc(\mc I)}$ induces a functor $\tilde{Q}_{\mc I}$ which can be identified with the Gabriel quotient of $\hat{A}\text{-}\Mod_{\mc A(T)}$ by the subcategory $\Loc(U_{\hat{A}}^{-1}(\mc I)) = \Loc(\langle F_{\hat{A}}(\mc I) \rangle)$. Setting $\bar{A} \vcentcolon= Q_{\mc I}(A)$, altogether this yields a commutative diagram as follows:

\[\begin{tikzcd}
        {\mc A(\mc T)} & {\mc A(\mc T)} & {\frac{\mc A(\mc T)}{{\Loc(\mc I)}}} \\
        {\mc A(A\text{-}\Mod_{\mc T})} & {\hat{A}\text{-}\Mod_{\mc A( \mc T)}} \\
        {\frac{\mc A(A\text{-}\Mod_{\mc T})}{\Loc(\langle \hat{F}_A(\mc I)\rangle)}} & {\frac{\hat{A}\text{-}\Mod_{\mc A( \mc T)}}{\Loc(U_{\hat{A}}^{-1}(\mc I))}} & {\bar{A}\text{-} \Mod_{\frac{\mc A(\mc T)}{{\Loc(\mc I)}}}}
        \arrow[equals, from=1-1, to=1-2]
        \arrow["{\hat{F}_A}"', shift right=3, from=1-1, to=2-1]
        \arrow["Q_{\mc I}", from=1-2, to=1-3]
        \arrow["{F_{\hat{A}}}"', shift right=3, from=1-2, to=2-2]
        \arrow["{F_{\bar{A}}}"', shift right=3, from=1-3, to=3-3]
        \arrow["{\hat{U}_A}"', shift right=3, from=2-1, to=1-1]
        \arrow["E"', from=2-1, to=2-2]
        \arrow["\cong", draw=none, from=2-1, to=2-2]
        \arrow[two heads, from=2-1, to=3-1]
        \arrow["{U_{\hat{A}}}"', shift right=3, from=2-2, to=1-2]
        \arrow[two heads, from=2-2, to=3-2]
        \arrow["{\tilde{Q}_{\mc I}}"{description}, from=2-2, to=3-3]
        \arrow["\cong", from=3-1, to=3-2]
        \arrow["\cong", from=3-2, to=3-3]
        \arrow["{U_{\bar{A}}}"', shift right=3, from=3-3, to=1-3]
        \arrow["G"', from=1-3, to=3-1, rounded corners,
        to path={ -- ([yshift=2ex]\tikztostart.north)
            -| ([xshift=-2ex]\tikztotarget.west)        [near end]\tikztonodes
            |- ([xshift=-2ex]\tikztotarget)}]
    \end{tikzcd}\]
This proves the statement. 
\end{proof}

To prove a converse to \Cref{left to right}, we need a sequence of lemmas. 
\begin{lemma}\label{intersection_of_homological_primes_is_zero}
    Let $F\colon \mc T \to \mc S$ be a finite separable extension of finite degree. For every homological prime $\mc H \in \Spc^{\mathrm{h}}(\mc T^c)$, the intersection of all homological primes in the fibre $(\varphi^h_F)^{-1}\{ \mc H\} \subseteq \Spc^{\mathrm{h}}(\mc S^c)$ is equal to the Serre $\otimes$-ideal $\langle \hat{F}(\mc H) \rangle$ of $\mc A(\mc S)^{\fp}$. 
\end{lemma}

\begin{proof}
    Let $A \in \mc T^c$ be a separable ring of finite degree such that $\mc S \cong A\text{-}\Mod_{\mc T}$ and under which $F$ becomes extension-of-scalars $F_A$. By \Cref{compatibilities}, it follows that for every homological prime $\mc H \in \Spc^{\mathrm{h}}(\mc T^c)$, 
    \begin{equation*}
        \frac{\mc A (\mc S)}{\Loc(\langle \hat{F}(\mc H) \rangle )} \cong \bar{A}\text{-}\Mod_{\mc A(\mc T)_{\mc H}},
    \end{equation*}
    where $\bar{A} \vcentcolon= Q_{\mc H}(\hat{A}) \in \mc A(\mc T)_{\mc H}$.
    Therefore, we may alternatively prove, for every homological prime $\mc H \in \Spc^{\mathrm{h}}(\mc T^c)$, that the intersection of all maximal Serre $\otimes$-ideals of $\left(\bar{A}\text{-}\Mod_{\mc A(\mc T)_{\mc H}} \right)^{\fp}$ is equal to $0$. We will show this by induction on the degree of $A$.  
    
    First we note, regardless of the degree of $A$, if $\pi_{\mc T}(\mc H) \notin \supp_{\mc T}(A)$, then $\bar{A} = 0 \in \mc A(\mc T)_{\mc H}$ and hence $\bar{A}\text{-}\Mod_{\mc A(\mc T)_{\mc H}} =0$. So, in this case, $\bar{A}\text{-}\Mod_{\mc A(\mc T)_{\mc H}} =0$ has no maximal Serre $\otimes$-ideals to speak of, making the statement vacuously true.

    If $\deg_{\mc T}(A) =1$, by \cite[Proposition 4.1]{finitedegree}, there is a decomposition of `big' tt-categories
    \begin{equation*}
        \mc T \cong \mc T_1 \times \mc T_2,
    \end{equation*}
    under which $A$ corresponds to $(\unit, 0) \in \mc T_1 \times \mc T_2$. This means that 
    \begin{equation}\label{bijection}
        \Spc^{\mathrm{h}}(\mc T^c) \cong \Spc^{\mathrm{h}}(\mc T^c_1) \sqcup \Spc^{\mathrm{h}}(\mc T^c_2).
    \end{equation}
    Therefore, if $\mc H \in \Spc^{\mathrm{h}}(\mc T^c )$ is a homological prime with $\pi_{\mc T}(\mc H) \in \supp_{\mc T}(A)$, it must correspond to a homological prime $\mc H_1$ of $\mc T_1$ under the bijection \eqref{bijection}. This implies that 
    \begin{equation*}
      \mc A(\mc T)_{\mc H} \cong \mc A(\mc T_1)_{\mc H_1} \cong \bar{A}\text{-}\Mod_{\mc A(\mc T)_{\mc H}},
    \end{equation*}
    because $\bar{A}$ corresponds to $\bar{\unit}$ under $\mc A(\mc T)_{\mc H} \cong \mc A(\mc T_1)_{\mc H_1}$. By construction, the only maximal Serre $\otimes$-ideal of $\mc A(\mc T)_{\mc H}$ is $(0)$, so the statement follows. 

    Now, suppose that $\deg_{\mc T}(A) > 1$ and $\mc H \in \Spc^{\mathrm{h}}(\mc T^c)$ is a homological prime with $\pi_{\mc T}(\mc H) \in \supp_{\mc T}(A)$. Let $\mc G \in \Spc^{\mathrm{h}}(\mc S^c)$ be a homological prime with $\varphi_F^h(\mc G)=\mc H$ (such a $\mc G$ exists precisely because $\pi_{\mc T}(\mc H) \in \supp_{\mc T}(A)$, see \cite[Theorem 5.12]{bigsupport}). Then $F\colon \mc T \to \mc S$ induces a functor on homological residue fields:
    \begin{equation*}
       \bar{F} \colon \mc A(\mc T)_{\mc H} \to \mc A(\mc S)_{\mc G}. 
    \end{equation*}
    Note that $\bar{F}$ is exact, symmetric monoidal, colimit-preserving and sends finitely presented objects to finitely presented objects. 
    If we set $B \vcentcolon= F(A)$, then $\bar{B}\vcentcolon= Q_{\mc G}(\hat{B})=\bar{F}(\bar{A})$ in~$ \mc A(\mc S)_{\mc G}$ and by \Cref{sep ext in abelian cats}, the functor $\bar{F}$ induces a commutative diagram
    \begin{equation}\label{comm diagram}
\begin{tikzcd}
	{\mc A(\mc T)_{\mc H}} & {\mc A(\mc S)_{\mc G}} \\
	{\bar{A}\text{-}\Mod_{\mc A(\mc T)_{\mc H}}} & {\bar{B}\text{-}\Mod_{\mc A(\mc S)_{\mc G}} \rlap{ .}}
	\arrow["{\bar{F}}", from=1-1, to=1-2]
	\arrow["{F_{\bar{A}}}"', shift right=3, from=1-1, to=2-1]
	\arrow["{U_{\bar{B}}}", shift left=3, from=1-2, to=2-2]
	\arrow["{U_{\bar{A}}}"', shift right=3, from=2-1, to=1-1]
	\arrow["{\tilde{F}}", from=2-1, to=2-2]
	\arrow["{F_{\bar{B}}}", shift left=3, from=2-2, to=1-2]
\end{tikzcd}
    \end{equation}
    Now, by \cite[Theorem 3.6]{finitedegree}, we have $B \cong \unit \times C \in \mc S$ for some separable ring $C \in \mc S^c$ with $\deg_{\mc S}(C) = \deg_{\mc T}(A) -1$. Therefore, we get a decomposition
   \begin{equation*}
       \bar{B}\text{-}\Mod_{\mc A(\mc S)_{\mc G}} \cong \mc A(\mc S)_{\mc G} \times \bar{C}\text{-}\Mod_{\mc A(\mc S)_{\mc G}}. 
   \end{equation*}
   By the induction hypothesis, the intersection of all maximal Serre $\otimes$-ideals of $\left(\bar{C}\text{-}\Mod_{\mc A(\mc S)_{\mc G}}\right)^{\fp}$ is $0$ and thus we conclude the same for $\left(\bar{B}\text{-}\Mod_{\mc A(\mc S)_{\mc G}}\right)^{\fp}$.

   Next, let $I$, respectively $J$, denote the set of all maximal Serre $\otimes$-ideals of $\left(\bar{A}\text{-}\Mod_{\mc A(\mc S)_{\mc G}}\right)^{\fp}$, respectively $\left(\bar{B}\text{-}\Mod_{\mc A(\mc S)_{\mc G}}\right)^{\fp}$. We claim that 
   \begin{equation}\label{inclusion}
       \tilde{F}\left( \bigcap_{\mc E \in I} \mc E \right) \subseteq \bigcap_{\mc E' \in J} \mc E' = 0,
   \end{equation}
   To see this, pick $\mc E' \in J$ and consider the functor $H_{\mc E'}$, defined as the composition
   \begin{equation*}
       \mc T \longrightarrow \bar{A}\text{-}\Mod_{\mc A(\mc T)_{\mc H}} \xlongrightarrow{\tilde{F}} \bar{B}\text{-}\Mod_{\mc A(\mc S)_{\mc G}} \longrightarrow \frac{\bar{B}\text{-}\Mod_{\mc A(\mc S)_{\mc G}}}{\Loc(\mc E')}.
   \end{equation*}
   Then $H_{\mc E'}$ satisfies the conditions of \cite[Proposition 5.5]{bigsupport}. It follows that the intersection $\Ker(\hat{H}_{\mc E'}) \cap \,\mc A(\mc T)^{\fp}$ is a homological prime of $\mc T$ that corresponds exactly to the Serre $\otimes$-ideal 
   \begin{equation*}
       \tilde{F}^{-1}(\mc E') \cap \left(\bar{A}\text{-}\Mod_{\mc A(\mc T)_{\mc H}}\right)^{\fp},
   \end{equation*}
   which is therefore a maximal Serre $\otimes$-ideal of $\left(\bar{A}\text{-}\Mod_{\mc A(\mc S)_{\mc G}}\right)^{\fp}$. This readily implies the inclusion \eqref{inclusion}. 
   
    To finish, by the commutativity of the diagram (\ref{comm diagram}), equality \eqref{inclusion} implies that 
    \begin{equation*}
      \bar{F}U_{\bar{A}}\left( \bigcap_{\mc E\in I} \mc E \right) =  U_{\bar{B}}\tilde{F}\left( \bigcap_{\mc E \in I} \mc E\right) =0 .
    \end{equation*}
   As the functors $\bar{F}$ and $U_{\bar{A}}$ are conservative (see \Cref{homological spectrum} for $\bar{F}$), it follows that
   \begin{equation*}
       \bigcap_{\mc E \in I} \mc E =0, 
   \end{equation*}
   as desired.
\end{proof}

\begin{construction}\label{construction}
    Given finitely many homological primes $\mc H_1, \dots, \mc H_n \in \Spc^{\mathrm{h}}(\mc T^c)$ of a `big' tt-category $\mc T$, we consider the functor 
    \begin{equation*}
     Q_{\mc H_1, \dots, \mc H_n} \vcentcolon= \prod_{i} Q_{\mc H_i}   \colon \mc A(\mc T) \to \prod_{i=1}^n \mc A(\mc T)_{\mc H_i}.
    \end{equation*}
    The kernel of this functor is exactly $ \Loc(\cap_{i=1}^n\mc H_i)$: since $Q_{\mc H_1, \dots, \mc H_n} $ is an exact functor that preserves colimits and finitely presented objects, it follows by \cite[Proposition A.6]{BKSruminations} that
    \begin{multline}\label{intersection and loc commute}
        \Ker(Q_{\mc H_1, \dots, \mc H_n})  = \Loc\left( \Ker(Q_{\mc H_1, \dots, \mc H_n}) \cap \mc A(\mc T)^{\fp} \right) \\   =  \Loc\left( \bigcap_{i=1}^n \Ker(Q_{\mc H_i}) \cap \mc A(\mc T)^{\fp} \right) = \Loc\left( \bigcap_{i=1}^n \mc H_i \right). 
    \end{multline}
    Therefore, $Q_{\mc H_1, \dots, \mc H_n}$ induces a functor $\tilde{Q}_{\mc H_1, \dots, \mc H_n}$ making the following diagram commute (here $L$ denotes the quotient functor)
\[\begin{tikzcd}
	{\mc A(\mc T) } && \\
	{\frac{\mc A(\mc T)}{\Loc(\cap_{i=1}^n \mc H_i)}} && {\prod_{i=1}^n \mc A(\mc T)_{\mc H_i}\rlap{ .}}
	\arrow["L"', two heads, from=1-1, to=2-1]
	\arrow["{Q_{\mc H_1, \dots, \mc H_n}}", from=1-1, to=2-3]
	\arrow["{\tilde{Q}_{\mc H_1, \dots, \mc H_n}}"', from=2-1, to=2-3]
\end{tikzcd}\]
\end{construction}

\begin{lemma}\label{decomposition}
    Let $\mc T$ be a `big' tt-category and $\mc H_1, \dots, \mc H_n \in \Spc^{\mathrm{h}}(\mc T^c)$. Suppose that $\pi_{\mc T}(\mc H_1), \dots, \pi_{\mc T}(\mc H_n)$ are distinct, closed, visible points of $\Spc(\mc T^c)$. Then the functor (\Cref{construction})
    \begin{equation*}
        \tilde{Q}_{\mc H_1, \dots, \mc H_n} \colon \frac{\mc A(\mc T)}{\Loc(\cap_{i=1}^n \mc H_i)} \to \prod_{i=1}^n \mc A(\mc T)_{\mc H_i} .
    \end{equation*}
    is an equivalence. 
\end{lemma}
\begin{proof}
To show that $\tilde{Q}\vcentcolon= \tilde{Q}_{\mc H_1, \dots, \mc H_n}$ is an equivalence we construct a quasi-inverse for it. For this, set $\mc M_i \vcentcolon= \pi_{\mc T}(\mc H_i) \in \Spc(\mc T^c)$, and consider the idempotent triangles (in the sense of \cite{Balmer-Favi})
\begin{equation}\label{idempotent triangle}
    \mb E_{i} \to \unit \to \mb F_{i} \to \Sigma \mb E_{i},
\end{equation}
corresponding to the closed Thomason subsets $\{ \mc M_i\} \subseteq \Spc(\mc T^c)$. Recall that $$\Ker(\mb F_{i} \otimes -) = \Loc\left(\left\{ a \in \mc T^c \mid \supp_{\mc T}(a) \subseteq \{ \mc M_i\}\right\}\right) \subseteq \mc T.$$
We claim that 
\begin{equation}\label{orthogonal idempotents}
   Q_{\mc H_j} (\hat{\mb E}_i)=\begin{cases}
			\bar{\unit}, &  i = j \\
            0, & i \neq j
		 \end{cases}
\end{equation}
When $i = j$, it suffices to prove that $Q_{\mc H_i}(\hat{\mb F}_i) = 0$ because of the idempotent triangle~\eqref{idempotent triangle}. As $Q_{\mc H_i}(\hat{\mb F}_i)$ is flat in $\mc A(\mc T)_{\mc H_i}$, we have that 
\begin{equation*}
   \mc I_i = \Ker(Q_{\mc H_i}(\hat{\mb F}_i) \otimes -) \cap \mc A(\mc T)_{\mc H_i}^{\fp},
\end{equation*}
is a Serre $\otimes$-ideal of $\mc A(\mc T)_{\mc H_i}^{\fp}$. Therefore $\mc I_i$ must equal $(0)$ or $\mc A(\mc T)_{\mc H_i}^{\fp}$ itself. Since $\mc M_i$ is a visible point, there exists an object $a_i \in \mc T^c$ with $\supp_{\mc T}(a_i) = \{ \mc M_i\}$. 
By definition of the idempotent $\mb F_{i}$, it follows that $\mb F_i \otimes a_i =0$ and so $ 0 \neq Q_{\mc H_i}(\hat{a}_i) \in \mc I_i$. Thus, we find that $\bar{\unit} \in \mc A(\mc T)_{\mc H_i}^{\fp}= \mc I_i  $ and $Q_{\mc H_i}(\hat{\mb F}_i) = Q_{\mc H_i}(\hat{\mb F}_i) \otimes \bar{\unit} = 0$ as desired. 
When $i \neq j$, we note that by \cite[Theorem 5.18]{Balmer-Favi}
\begin{equation*}
    \mb E_i \otimes \mb E_j \cong \mb E_{\{\mc M_i\} \cap \{\mc M_j\} } = \mb E_{\varnothing} = 0, 
\end{equation*}
since the $\mc M_i$ are distinct by assumption. It follows that 
\begin{equation*}
    Q_{\mc H_j} (\hat{\mb E}_i) = Q_{\mc H_j} (\hat{\mb E}_i) \otimes \bar{\unit} = Q_{\mc H_j} (\hat{\mb E}_i) \otimes Q_{\mc H_j} (\hat{\mb E}_j) =Q_{\mc H_j} (\hat{\mb E}_i \otimes \hat{\mb E}_j)  = 0. 
\end{equation*}

Now, fix $i \in \{1, \dots, n \}$. It follows from equality \eqref{orthogonal idempotents} that, for every $j \neq i$,
\begin{equation*}
    \hat{\mb E}_i \otimes \Loc(\mc H_i) \subseteq \Loc(\mc H_j). 
\end{equation*}
Since we certainly also have 
\begin{equation*}
    \hat{\mb E}_i \otimes \Loc(\mc H_i) \subseteq \Loc(\mc H_i), 
\end{equation*}
it follows that, 
\begin{equation*}
    \hat{\mb E}_i \otimes \Loc(\mc H_i) \subseteq \bigcap_{j=1}^n \Loc(\mc H_j) \overset{\eqref{intersection and loc commute}}{=} \Loc\left(\bigcap_{j=1}^n \mc H_j\right). 
\end{equation*}
Therefore, by the universal property of the Gabriel quotient, we get a commutative diagram
\[\begin{tikzcd}
	{\mc A(\mc T)} & {\mc A(\mc T)} \\
	{\mc A(\mc T)_{\mc H_i}} & {\frac{\mc A(\mc T)}{\Loc(\cap_{i=1}^n \mc H_i)} \rlap{ .}}
	\arrow["{G_i \vcentcolon= \hat{\mb E}_i \otimes -}", from=1-1, to=1-2]
	\arrow["{Q_{\mc H_i}}"', two heads, from=1-1, to=2-1]
	\arrow["L", two heads, from=1-2, to=2-2]
	\arrow["{\exists \tilde{G}_i}"', from=2-1, to=2-2]
\end{tikzcd}\]
Now, observe that, by \eqref{orthogonal idempotents}, 
\begin{equation*}
    Q_{\mc H_1, \dots, \mc H_n}\left(\bigoplus_{i=1}^n \hat{\mb E}_i\right) \cong \big( \bar{\unit} \big)_{j=1}^n \in \prod_{j=1}^n \mc A(\mc T)_{\mc H_j},
\end{equation*}
Since $\Ker(Q_{\mc H_1, \dots, \mc H_n}) = \Ker(L)$ by \eqref{intersection and loc commute}, it also follows that
\begin{equation*}
    L\left(\bigoplus_{i=1}^n \hat{\mb E}_i\right) \cong \bar{\unit} \in \frac{\mc A(\mc T)}{\Loc(\cap_{i=1}^n \mc H_i)}.
\end{equation*}
Using this, it is straightforward to verify that 
\begin{equation*}
   \bigoplus_{i=1}^n \tilde{G}_i \colon  \prod_{i=1}^n \mc A(\mc T)_{\mc H_i} \to \frac{\mc A(\mc T)}{\Loc(\cap_{i=1}^n \mc H_i)}
\end{equation*}
is the sought-after quasi-inverse for $\tilde{Q}$. 
\end{proof}

\begin{lemma}\label{cor}
    Let $F\colon \mc T \to \mc S$ be a finite separable extension of finite degree and let $\mc G_1, \dots, \mc G_n \in \Spc^{\mathrm{h}}(\mc S^c)$. Suppose that $\pi_{\mc S}(\mc G_1), \dots, \pi_{\mc S}(\mc G_1)$ are distinct, visible points of $\Spc(\mc S^c)$ that all lie in the same fibre of $\varphi_{F}$. Then the functor (\Cref{construction})
    \begin{equation*}
       \tilde{Q}_{\mc G_1, \dots, \mc G_n} \colon \frac{\mc A(\mc S)}{\Loc(\cap_{i=1}^n \mc G_i)} \to \prod_{i=1}^n \mc A(\mc S)_{\mc G_i} 
    \end{equation*}
    is an equivalence. 
\end{lemma}
\begin{proof}
    Set $\mc Q_i \vcentcolon = \pi_{\mc S}(\mc G_i) \in \Spc(\mc S^c)$ for $ i =1, \dots, n$, and let $\mc P \in \Spc(\mc T^c)$ denote the point such that for all $i$
    \begin{equation*}
        \varphi_{F}(\mc Q_i) = \mc P. 
    \end{equation*}
    Consider the `big' tt-category 
    \begin{equation*}
        \mc S' = \frac{\mc S}{\Loc(\langle F(\mc P \rangle)},
    \end{equation*}
    where $\langle F(\mc P \rangle)$ is the thick $\otimes$-ideal of $\mc S^c$ generated by $F(\mc P)$.
    Since $F(\mc P) \subseteq \mc Q_i$ for all $i$, the homological primes $\mc G_1, \dots, \mc G_n$ of $\mc S$ correspond to homological primes $\mc G_1', \dots, \mc G_n'$ of $\mc S'$, see \Cref{homological_spectrum_of_ttquotients}. Moreover,  this correspondence is compatible with the homological residue fields in such a way that we get a commutative diagram 
\[\begin{tikzcd}
	{\frac{\mc A(\mc S)}{\Loc(\cap_{i=1}^n \mc G_i)}} && {\prod_{i=1}^n \mc A(\mc S)_{\mc G_i}} \\
	{\frac{\mc A(\mc S')}{\Loc(\cap_{i=1}^n \mc G_i')}} && {\prod_{i=1}^n \mc A(\mc S')_{\mc G_i'} \rlap{ .}}
	\arrow["{\tilde{Q}_{\mc G_1, \dots, \mc G_n}}", from=1-1, to=1-3]
	\arrow["\cong"', from=1-1, to=2-1]
	\arrow["\cong", from=1-3, to=2-3]
	\arrow["{\tilde{Q}_{\mc G_1', \dots, \mc G_n'}}", from=2-1, to=2-3]
\end{tikzcd}\]
    So, it suffices to show that $\tilde{Q}_{\mc G_1', \dots, \mc G_n'}$ is an equivalence. 
    
    Since the $\pi_{\mc S}(\mc G_i)$ are all distinct and visible, the same is true for the $\pi_{\mc S'}(\mc G_i')$. We claim that the points $\pi_{\mc S'}(\mc G_i')$ are closed. To see this, first note that the fibres of $\varphi_F$ are discrete by \cite[Theorem 3.10]{BalmerQuillenstrat}, since $F$ is a finite separable extension of finite degree. Now suppose there exists a point $\mc Q \in \Spc(\mc S^c)$ with 
    \begin{equation*}
        F(\mc P) \subseteq \mc Q \subseteq \pi_{\mc S}(\mc G_i). 
    \end{equation*}
    This implies that
    \begin{equation*}
        \mc P \subseteq \varphi_F(Q) \subseteq \varphi_F\left( \pi_{\mc S}(\mc G_i)\right) = \mc P,
    \end{equation*}
    i.e.\@ $\mc Q$ lies in the fibre over $\mc P$. Therefore, we must have $\mc Q = \pi_{\mc S}(\mc G_i)$ by discreteness of the fibre. In particular, this means that the points $\pi_{\mc S'}(\mc G_i')$ are all closed.

    So, the $\pi_{\mc S'}(\mc G_1'), \dots, \pi_{\mc S'}(\mc G_1')$ are distinct, closed, visible points. Therefore, \Cref{decomposition} applies and we conclude that $\tilde{Q}_{\mc G_1', \dots, \mc G_n'}$ is an equivalence.     
\end{proof}

\begin{proposition}\label{right to left}
    Let $F\colon \mc T \to \mc S$ be a finite separable extension of finite degree and let $\mc P \in \Spc(\mc T^c)$ with $\mc P \in \im(\varphi_{F})$. Suppose that for every $\mc Q \in \varphi_{F}^{-1}(\{\mc P\})$ the set $\pi_{\mc S}^{-1}(\{\mc Q\})$ is a singleton and  that $\mc Q$ is a visible point of $\Spc(\mc S^c)$. 
    If $\mc T$ is residually regular at $ \mc P$, then $\mc S$ is residually regular at $\mc Q$ for all $\mc Q \in \varphi_{F}^{-1}(\{\mc P\})$.
\end{proposition}
\begin{proof}
Let $\mc H \in \Spc^{\mathrm{h}}(\mc T^c)$ be the unique homological prime lying over $\mc P$ (see \Cref{residual regularity implies nerves of steel}). Since the extension is of finite degree, the fibre of $\varphi_{F}$ over $\mc P$ is finite by \cite[Theorem 3.10]{BalmerQuillenstrat}, say $$ \varphi_{F}^{-1}(\{\mc P\}) = \{ \mc Q_1, \dots, \mc Q_d\}.$$ 
By assumption, there is exactly one homological prime $\mc G_i \in \Spc^{\mathrm{h}}(\mc S^c)$ lying over $\mc Q_i$ for every $i \in \{ 1, \dots, d\}$. Note that this implies that the fibre of $\varphi_{F}^h$ over $\mc H$ is exactly
\begin{equation*}
   (\varphi_{F}^h)^{-1}(\{\mc H\})= \{\mc G_1, \dots, \mc G_d \}.
\end{equation*}
Therefore, for every $i \in \{1, \dots, d\}$, the functor $F$ induces an exact symmetric monoidal functor on homological residue fields
\begin{equation*}
    \bar{F}_i \colon \mc A(\mc T)_{\mc H} \to \mc A(\mc S)_{\mc G_i}.
\end{equation*}
Bundling these functors together, we obtain a functor 
\begin{equation*}
 \bar{F} \vcentcolon= \prod_{i=1}^d \bar{F}_i \colon \mc A(\mc T)_{\mc H} \to  \prod_{i=1}^d \mc A(\mc S)_{\mc G_i},
\end{equation*}
which fits into the following commutative diagram (see \Cref{construction}):
\[\begin{tikzcd}
	{\mc A(\mc T)} && {\mc A(\mc T)_{\mc H}} \\
	{\mc A(\mc S)} && {\prod_{i=1}^d \mc A(\mc S)_{\mc G_i} \rlap{ .}}
	\arrow["{Q_{\mc H}}", two heads, from=1-1, to=1-3]
	\arrow["{\hat{F}}"', from=1-1, to=2-1]
	\arrow["{\bar{F}}", from=1-3, to=2-3]
	\arrow["{Q_{\mc G_1, \dots, \mc G_d}}"', from=2-1, to=2-3]
\end{tikzcd}\]
We claim that $\bar{F}$ preserves injectives. To see this, note that \Cref{intersection_of_homological_primes_is_zero,cor} imply that the above commutative diagram can be extended as follows 
\[\begin{tikzcd}
	{\mc A(\mc T)} & {\mc A(\mc T)_{\mc H}} \\
	{\mc A(\mc S)} & {\prod_{i=1}^d \mc A(\mc S)_{\mc G_i}} \\
	{\frac{\mc A(\mc S)}{\Loc(\langle \hat{F}(\mc H)\rangle)}} & {\frac{\mc A(\mc S)}{\Loc(\cap_{i=1}^d \mc G_i)} \rlap{ ,}}
	\arrow["{Q_{\mc H}}", two heads, from=1-1, to=1-2]
	\arrow["{\hat{F}}"', from=1-1, to=2-1]
	\arrow["{\bar{F}}", from=1-2, to=2-2]
	\arrow["{Q_{\mc G_1, \dots, \mc G_d}}"', from=2-1, to=2-2]
	\arrow[two heads, from=2-1, to=3-1]
	\arrow[equals, from=3-1, to=3-2]
	\arrow["\cong", from=3-2, to=2-2]
    \arrow["G"', from=1-2, to=3-1, rounded corners,
        to path={ -- ([yshift=2ex]\tikztostart.north)
            -| ([xshift=-2ex]\tikztotarget.west)        [near end]\tikztonodes
            |- ([xshift=-2ex]\tikztotarget)}]
\end{tikzcd}\]
where the functor $G$ is induced by the universal property of $Q_{\mc H}$. By \Cref{compatibilities}, the functor $G$ identifies with an extension-of-scalars functor with respect to a rigid and separable ring, and thus $G$ preserves injectives by \Cref{sep ext in abelian cats}. Therefore, $\bar{F}$ must preserve injectives as well.  

Now, consider the injective hull $\bar{\eta}=\bar{\eta}_{\mc H} \colon \bar{\unit} \to \overbar{\kappa(\mc H)}$ in $\mc A(\mc T)_{\mc H}$. Then $\bar{F}(\bar{\eta})$ is a monomorphism of the unit of $\prod_{i}\mc A(\mc S)_{\mc G_i}$ to the injective object $\bar{F}\left(\overbar{\kappa(\mc H)}\right)$. 
Therefore, the injective hull of the unit in $\prod_{i}\mc A(\mc S)_{\mc G_i}$, which is exactly
\begin{equation*}
    \left(\overbar{\kappa(\mc G_i)}\right)_{i=1}^d \in \prod_{i=1}^d \mc A(\mc S)_{\mc G_i},
\end{equation*}
is a direct summand of $\bar{F}(\overbar{\kappa(\mc H)}) = \overbar{F(\kappa(\mc H))}$. In other words, $\overbar{\kappa(\mc G_i)}$ is a direct summand of $\overbar{F(\kappa(\mc H))}$ for every $i$. This decomposition already holds in $\mc S$, by Krause's isomorphisms~\eqref{injective iso}. For every $i$, the functor $F$ induces the following commutative diagram of geometric tt-functors:
\[\begin{tikzcd}
	{\mc T} & {\mc T_{\mc P}} \\
	{\mc S} & {\mc S_{\mc Q_i} \rlap{ .}}
	\arrow[two heads, from=1-1, to=1-2]
	\arrow["F"', from=1-1, to=2-1]
	\arrow["{\tilde{F}_i}", from=1-2, to=2-2]
	\arrow[two heads, from=2-1, to=2-2]
\end{tikzcd}\]
It follows that $\kappa(\mc G_i)$ is a direct summand of $\tilde{F}_i(\kappa(\mc H))$ in $\mc S_{\mc Q_i}$. Since $\mc T$ is residually regular at $\mc P$ by assumption, we have that $\kappa(\mc H)$ is compact in the stalk category $\mc T_{\mc P}$ and thus $\tilde{F}_i(\kappa(\mc H))$ is compact in $\mc S_{\mc Q_i}$. So $\kappa(\mc G_i)$ is compact in $\mc S_{\mc Q_i}$ and as such $\mc S$ is residually regular at $\mc Q_i$ for all $i$. 
\end{proof}

\begin{corollary}\label{the big one}
    Let $F\colon \mc T \to \mc S$ be a finite separable extension of finite degree and let $\mc Q \in \Spc(\mc S^c)$. Suppose that $\mc S$ satisfies the steel condition and that $\Spc(\mc S^c)$ is Noetherian. Then $\mc S$ is residually regular at $\mc Q$ if and only if $\mc T$ is residually regular at $ \mc P \vcentcolon= \varphi_{F}(\mc Q)$. 
\end{corollary}
\begin{proof}
    The forward direction is taken care of by \Cref{left to right}. After all, as $F$ is a finite separable extension, its right adjoint $U$ preserves compacts. Moreover, since the separable extension is of finite degree, the fibres of $\varphi_F$ are discrete by \cite[Theorem 3.10]{BalmerQuillenstrat}. In particular, every point of $\Spc(\mc S^c)$ is closed in its fibre. Finally, all points are visible as $\Spc(\mc S^c)$ is assumed to be Noetherian.
    
    For the converse, we apply \Cref{right to left} (all conditions are readily satisfied by the assumption that $\mc S$ satisfies the steel condition and $\Spc(\mc S^c)$ is Noetherian). 
\end{proof}

\begin{remark}
    During the writing of this paper, two other results \cite{regandsepmonads, singcatsepext} regarding separable extensions were posted on the arxiv. In \cite{regandsepmonads}, a similar descent-ascent statement is proved for regularity of $k$-linear triangulated categories\footnote{Note that their definition of regularity implies that the compact objects are $\Hom$-finite over $k$.} as proposed in \cite{kostas}, where $k$ is a Noetherian commutative ring; cf.\@ \Cref{inthelit}. They prove that this notion of regularity descends and ascends via extensions along certain `nice' separable monads on the triangulated category. In \cite{singcatsepext}, it is shown that a finite étale morphism of schemes gives rise to a finite separable extension at the level of the singularity categories. From this one can deduce the usual algebro-geometric descent-ascent for finite étale morphisms. 
    We remark that both these results do not subsume ours, nor vice versa. 
\end{remark}

\section{Permutation modules}\label{mod rep sec}
In this section, we apply the above results in modular representation theory. Specifically, we consider the residual regularity of the derived category of permutation $kG$-modules, as studied in \cite{TT-perm}.

\begin{hypothesis}\label{finite group hypothesis}
    In the following, $G$ will be a finite group and $k$ a field of positive characteristic $p$ dividing the order of $G$. 
\end{hypothesis}

We start by recalling some elements from \cite{TT-perm}. 
\begin{recollection}
    For a $G$-set $X$, we let $k(X)$ denote the free $kG$-module with basis $X$ and $kG$-action given by $k$-linearly extending the $G$-action on $X$. A permutation $kG$-module is a $kG$-module isomorphic to one of the form $k(X)$. These modules form a full additive subcategory $\text{Perm}(G;k)$ of $\Mod(kG)$. The usual tensor product of $\Mod(kG)$, tensoring over $k$ with diagonal $G$-action, makes $\text{Perm}(G;k)$ into a tensor category. Then, the `big' derived category of permutation $kG$-modules is defined as the localisation
\begin{equation*}
   \DPerm(G;k) \vcentcolon= \K(\text{Perm}(G;k))[\{G\text{-quasi-isos}\}^{-1}], 
\end{equation*}
where a $G$-quasi-isomorphism $f: X \to Y$ is a morphism of complexes such that the induced morphism on $H$-fixed points $f^H$ is a quasi-isomorphism for all subgroups $H \leq G$. Equivalently, $\DPerm(G;k)$ may be described as the localising subcategory of $\K(\text{Perm}(G;k))$ generated by $\K^b(\perm(G;k)^{\natural})$, where $\perm(G;k)$ is the full subcategory of finitely generated modules in $\text{Perm}(G;k)$ and $(-)^{\natural}$ denotes the idempotent completion. It follows that 
\begin{equation*}
   \DPerm(G;k)^c \cong \K^b(\perm(G;k)^{\natural}). 
\end{equation*}
A set of rigid-compact generators of $\DPerm(G;k)$ is given by the permutation modules $k(G/H)$ for $H \leq G$. 
\end{recollection}

\begin{example}\label{trivial group}
    For the trivial group $G=1$, the category $\DPerm(1;k) = \D(k)$ is exactly the derived category of $k$, with compact part $\D^b(k)$. 
\end{example}

\begin{notation}
    To abbreviate notation, we will sometimes write $\mc T(G)$ for $\DPerm(G;k)$, suppressing the field $k$. 
\end{notation}

Let us now recall some significant geometric tt-functors on these categories. 
\begin{recollection}
    For any subgroup $H \leq G$, restriction of scalars yields the usual functor 
\begin{equation*}
    \Res^G_H \colon \DPerm(G;k) \to \DPerm(H;k). 
\end{equation*}
Its right adjoint is given by induction and is denoted by $\Ind^G_H$. 
If $N \trianglelefteq G$ is a normal subgroup, inflation of modules along the quotient map $G \to G/N$ induces a functor denoted by 
\begin{equation*}
    \Infl^{G/N}_G \colon \DPerm(G/N;k) \to \DPerm(G;k).
\end{equation*}
More involved are the so-called modular fixed point functors 
\begin{equation*}
   \Psi^{H}= \Psi^{H;G}\colon \DPerm(G;k) \to \DPerm(G/\!\!/ H;k), 
\end{equation*}
defined for any $p$-subgroup $H \leq G$. Here $G/\!\!/H$ denotes the Weyl group $N_G(H)/H$. On the generators $k(G/K)$, for $K \leq G$, we have 
\begin{equation*}
    \Psi^{H;G}(k(G/K)) \cong k((G/K)^H) = \begin{cases}
			k(N_G(H, K)/K), & H \text{ is subconjugate to } K \text{ in } G,\\
            0, & \text{otherwise},
		 \end{cases}
\end{equation*}
where $N_G(H, K) = \{ g \in G \mid g^{-1}Hg \subseteq K\}$.
Note that when $H$ is a normal subgroup, this simplifies to $G/\!\!/ H = G/H$ and 
\begin{equation*}
    \Psi^{H;G}(k(G/K))= \begin{cases}
			k\left( (G/H)/(K/H) \right), & H \leq K,\\
            0, & H \nleq K.
		 \end{cases}
\end{equation*}
The right adjoint of modular fixed points $ \Psi^{H;G}$ is denoted by $ \Psi^{H;G}_{\rho}$. See \cite[Section 5]{TT-perm} for more details on modular fixed point functors. 
\end{recollection}

\begin{recollection}
    In \cite{TT-perm}, Balmer--Gallauer give a full description of the Balmer spectrum of $\DPerm(G;k)^c$. We shall only recall the relevant results here. There is a finite localisation of $\DPerm(G;k)$ to the homotopy category of injectives $\K\Inj(kG)$:
\begin{equation*}
   \Gamma \colon \DPerm(G;k) \to \K\Inj(kG).
\end{equation*}
On compacts, $\Gamma$ can be identified with the functor 
\begin{equation*}
    \K^b(\perm(G;k)^{\natural}) \to \D^b(kG)\vcentcolon= \D^b(\fgmod(kG)),
\end{equation*}
induced by the inclusion of $\perm(G;k)^{\natural}$ in $\fgmod(kG)$. Thus, at the level of Balmer spectra, $\Spc(\D^b(kG))$ can be found as a subset of $\Spc(\mc T(G)^c)$, which is moreover shown to be an open subset. It is called the cohomological open of $\Spc(\mc T(G)^c)$ and is denoted by $V_G$. Note that 
\begin{equation*}
    \Spc(\D^b(kG)) \cong \Spec^h(H^{\bullet}(G;k)),
\end{equation*}
where $H^{\bullet}(G;k)$ denotes group cohomology of $G$ with coefficients in $k$. 
As a reminder, the more familiar stable module category $\Stab(kG)$ can be found as a finite localisation of $\K\Inj(kG)$ and the induced map on spectra is exactly the inclusion of the open subset 
\begin{equation*}
    \Spc(\stab(kG)) \cong \Proj(H^{\bullet}(G;k)) \hookrightarrow \Spec^h(H^{\bullet}(G;k)) \cong \Spc(\D^b(kG)).  
\end{equation*}

The spectrum $\Spc(\mc T(G)^c)$ has one closed point $\mc M(H)$ for every conjugacy class of $p$-subgroups $H\leq G$. Specifically, $\mc M(H)$ is defined as the kernel of the functor 
\begin{equation*}
    \mc T(G)^c \xrightarrow{\Psi^{H;G}} \mc T(G/\!\!/H)^c \xrightarrow{\Res^{G/\!\!/H}_1} \D^b(k). 
\end{equation*}
The cohomological open $V_G$ only contains one closed point, for the trivial subgroup $H=1$. All other closed points $\mc M(H)$ for $H \neq 1$ are to be found in the complement of $V_G$. 

Finally, Balmer--Gallauer also show that $\Spc(\mc T(G)^c)$ is Noetherian and stratified in the sense of \cite{BHS}. By \cite[Theorem 4.7]{BHScomp}, it follows that $\mc T(G)$ satisfies the steel condition. 
\end{recollection}

\begin{proposition}\label{residue regularity of KInj}
    Under \Cref{finite group hypothesis}, the category $\DPerm(G;k)$ is residually regular at the closed point $\mc M(1)$. 
\end{proposition}
\begin{proof}
    As $\Psi^{1;G}$ is simply the identity functor on $\mc T(G)$, the point $\mc M(1)$ is exactly the kernel of the functor $\Res^G_1 \colon \mc T(G)^c \to \mc T(1)^c$. Since $\mc T(1) = \D(k)$ (\Cref{trivial group}) is a tt-field and $\Ind^G_1(\unit) = kG$ is compact in $\mc T(G)$, \Cref{going to tt-field} applies to the restriction functor $\Res^G_1$. Thus, the category $\mc T(G)$ is residually regular at $\mc M(1)$. 
\end{proof}

\begin{remark}\label{conjecture KInj is regular}
    The stalk category of $\mc T(G)$ at $\mc M(1)$ is exactly $\K\Inj(kG)$. So, \Cref{residue regularity of KInj} implies that the local category $\K\Inj(kG)$ is residually regular at its unique closed point. If we could answer the question raised in \Cref{residual regularity of local tt-categories} in the positive --- that is, it suffices to check residual regularity at closed points --- we could conclude that $\K\Inj(kG)$ is in fact residually regular at every point. It would follow that the stable module category $\Stab(kG)$ is residually regular as well. This is what one would expect since $\left(\K\Inj(kG)\right)^c$ and $\left(\Stab(kG)\right)^c$ are strongly generated triangulated categories; cf.\@ \Cref{inthelit}. 
\end{remark}

\begin{remark}
    In fact, the category $\DPerm(G;k)$ can be defined for any field $k$, with no restriction on the characteristic, and the same is true for the restriction functors~$\Res^G_H$. However, when the characteristic of $k$ does not divide $|G|$, the spectrum $\Spc(\mc T(G)^c)$ is just a point, by \cite[Corollary 3.9]{TT-perm}. More precisely, restriction to the trivial subgroup $\Res^G_1$ induces a surjective map on Balmer spectra in this case. Therefore, repeating the proof of \Cref{residue regularity of KInj}, $\DPerm(G;k)$ is always residually regular in the \emph{non-modular} setting.
\end{remark}

\begin{proposition}\label{restriction to subgroups}
    Under \Cref{finite group hypothesis}, let $H \leqslant G$ be a subgroup. Then the restriction functor $\Res^G_H \colon \DPerm(G;k) \to \DPerm(H;k)$ is a finite separable extension of finite degree. Consequently, $\DPerm(H;k)$ is residually regular at a point $\mc Q$ if and only if $\DPerm(G;k)$ is residually regular at $\varphi_{ \Res^G_H}(\mc Q)$. 
\end{proposition}
\begin{proof}
    By \cite[Recollection 4.6]{TT-perm}, $\Res^G_H$ is a finite separable extension with respect to the ring $A_H = k(G/H)$ with the usual unit $u\colon\unit \to R_H$ mapping $1$ to $\sum_{\gamma \in G/H} \gamma$ and multiplication $\mu \colon R_H \otimes R_H \to R_H$ that is characterised by $\mu(\gamma \otimes \gamma) = \gamma$ and $\mu(\gamma \otimes \gamma')=0$ for all $\gamma \neq \gamma'$ in $G/H$. This ring $A_H$ is of finite degree in $\DPerm(G;k)$ because it is of finite degree in $\text{Perm}(G;k)$ (bounded by the index $[G:H]$) and the degree does not increase along the tensor functor $\text{Perm}(G;k) \to \DPerm(G;k)$. Finally, as $\mc T(G)$ satisfies the steel condition and $\Spc(\mc T(G)^c)$ is Noetherian, the statement on residual regularity follows immediately by application of \Cref{the big one}.
\end{proof}

\begin{lemma}\label{about Psi(1)}
Under \Cref{finite group hypothesis}, suppose that $G$ is a $p$-group and consider the modular fixed points functor $\Psi^{G} = \Psi^{G;G}: \DPerm(G;k) \to \D(k)$. Suppose that there is an object $S \in  \DPerm(G;k)$ that satisfies, for $i \in \Z$, 
\begin{equation*}
     \Hom_{\mc T(G)}(k(G/H)[i], S) \cong \begin{cases}
			0, & H \lneq G \text{ or } i \neq 0, \\
            k, & H = G \text{ and } i=0. 
		 \end{cases}
\end{equation*}
and such that there exists $\sigma \colon \Psi^G(S) \to k$ in $\D(k)$ and $\zeta\colon \unit \to S$ in $\DPerm(G;k)$ with $\sigma \circ \Psi^G(\zeta) = \id_{k}$ in $\D(k)$.
Then $S$ is isomorphic to $\kappa(\mc M(G))$ in $\DPerm(G;k)$. 
\end{lemma}
\begin{proof}
Note that the closed point $\mc M(G)$ is exactly the kernel of $\Psi^G$ (restricted to compacts), since $\Res^{G/\!\!/G}_1 =\text{Id}$. Therefore, $\kappa(\mc M(G))$ is a direct summand of $\Psi^G_{\rho}(\mathbbm{1})$ in $\mc T(G)$ (see \Cref{functoriality_of_the_homological_spectrum}). The latter is indecomposable in $\mc T(G)$ as 
\begin{equation*}
    \Hom_{\mc T(G)}(k(G/H)[i], \Psi^G_{\rho}(\mathbbm{1})) \cong \begin{cases}
			0, & H \lneq G \text{ or } i \neq 0, \\
            k, & H = G \text{ and } i=0. 
		 \end{cases}
\end{equation*}
So $\kappa(\mc M(G)) \cong \Psi^G_{\rho}(\mathbbm{1})$ and it suffices to prove that $S$ is isomorphic to $\Psi^G_{\rho}(\mathbbm{1})$. 

By the adjunction $\Psi^G \dashv \Psi^G_{\rho}$ the given map $\sigma$ corresponds to a map $\chi: S \to \Psi^G_{\rho}(\mathbbm{1})$ for which
        \[\epsilon \circ \Psi^G(\chi) = \sigma,\]
        where $\epsilon$ is the counit of $\Psi^G \dashv \Psi^G_{\rho}$. We will prove that $\chi$ is an isomorphism. For this, it suffices to check that the morphisms
        \[\Hom_{\mc T(G)}\left(k(G/H)[i], \chi\right)\colon \Hom_{\mc T(G)}(k(G/H)[i], S) \to \Hom_{\mc T(G)}(k(G/H)[i], \Psi^G_{\rho}(\unit)) \] 
        are isomorphisms for all $H \leqslant G$ and for all $i \in \Z$ (because these are generators of $\mc T(G)$). 

        For $H \lneq G$, there is nothing to check because both source and target are zero by assumption for all $i \in \Z$. For $H = G$, i.e when $k(G/H) = \unit$, we must only check the case $i =0$ for the same reason. When $i =0$, both source and target are isomorphic to $k$. As such, the given map $\zeta \colon \unit \to S$ must generate all of $\Hom_{\mc T(G)}(\unit, S) \cong k$ and it suffices to prove that $\chi \circ \zeta \neq 0$, or equivalently that 
    \[\epsilon \circ \Psi^G(\chi \circ \zeta) \neq 0. \]
    We have 
    \begin{equation*}
        \epsilon \circ \Psi^G(\chi \circ \zeta) = \epsilon \circ \Psi^G(\chi) \circ \Psi^G(\zeta) = \sigma \circ \Psi^G(\zeta) = \id_k, 
    \end{equation*}
    where the final equality follows from our assumption on $\sigma$ and $\zeta$.
\end{proof}

\begin{example}\label{cyclic group of order p}
\hfill
\begin{enumerate} 
    \item Consider the following exact complex in $\DPerm(C_2;k)$:
   \begin{equation*}
       S= \quad \dots 0 \to 0 \to k \to kC_2 \to k \to 0 \to 0 \dots \in \DPerm(C_2), 
   \end{equation*}
concentrated in cohomological degrees $-2, -1$ and $0$. It is straightforward to check that $S$ satisfies the conditions of \Cref{about Psi(1)}, with maps 
\[\begin{tikzcd}
	\unit & \cdots & 0 & 0 & 0 & k & 0 & \cdots \\
	S & \cdots & 0 & k & {kC_2} & k & 0 & \cdots
	\arrow["\zeta"', from=1-1, to=2-1]
	\arrow["{=}"{description}, draw=none, from=1-1, to=2-2]
	\arrow[from=1-2, to=1-3]
	\arrow[from=1-3, to=1-4]
	\arrow[from=1-3, to=2-3]
	\arrow[from=1-4, to=1-5]
	\arrow[from=1-4, to=2-4]
	\arrow[from=1-5, to=1-6]
	\arrow[from=1-5, to=2-5]
	\arrow[from=1-6, to=1-7]
	\arrow["1", from=1-6, to=2-6]
	\arrow[from=1-7, to=1-8]
	\arrow[from=1-7, to=2-7]
	\arrow[from=2-2, to=2-3]
	\arrow[from=2-3, to=2-4]
	\arrow[from=2-4, to=2-5]
	\arrow[from=2-5, to=2-6]
	\arrow[from=2-6, to=2-7]
	\arrow[from=2-7, to=2-8]
\end{tikzcd}\]
and 
\[\begin{tikzcd}
	{\Psi^G(S) } & \cdots & 0 & k & 0 & k & 0 & \cdots \\
	{\unit  } & \cdots & 0 & 0 & 0 & k & 0 & \cdots
	\arrow["\sigma"', from=1-1, to=2-1]
	\arrow["{=}"{description}, draw=none, from=1-1, to=2-2]
	\arrow[from=1-2, to=1-3]
	\arrow[from=1-3, to=1-4]
	\arrow[from=1-3, to=2-3]
	\arrow[from=1-4, to=1-5]
	\arrow[from=1-4, to=2-4]
	\arrow[from=1-5, to=1-6]
	\arrow[from=1-5, to=2-5]
	\arrow[from=1-6, to=1-7]
	\arrow["1", from=1-6, to=2-6]
	\arrow[from=1-7, to=1-8]
	\arrow[from=1-7, to=2-7]
	\arrow[from=2-2, to=2-3]
	\arrow[from=2-3, to=2-4]
	\arrow[from=2-4, to=2-5]
	\arrow[from=2-5, to=2-6]
	\arrow[from=2-6, to=2-7]
	\arrow[from=2-7, to=2-8]
\end{tikzcd}\]
So we have $\kappa(\mc M(C_2)) \cong S$. Since $S$ is clearly compact in $\mc T(C_2)$, we conclude that $\DPerm(C_2;k)$ is residually regular at the closed point $\mc M(C_2)$. 

\item Suppose that $p$ is an odd prime. Using the identification $R \vcentcolon= k[t]/(t^p) \cong kC_p$, we may define maps $h, g$ and $f$ as follows: 
\begin{align*}
    & h : R \oplus k \xrightarrow{\begin{pmatrix} t & t^{p-1} \end{pmatrix}} R, \\
    & g :  R \oplus k \xrightarrow{\begin{pmatrix} t^{p-2} & t^{p-1} \\ -1 & 0 \end{pmatrix}} R \oplus k, \\ 
    & f : R \oplus k \xrightarrow{\begin{pmatrix} t & t^{p-1} \\ -1 & 0 \end{pmatrix}} R \oplus k.
\end{align*}
Then, consider the following eventually left periodic complex in $\DPerm(C_p;k)$:
\begin{equation*}
       S= \quad \dots kC_p \oplus k \xrightarrow{g} kC_p \oplus k \xrightarrow{f} kC_p \oplus k \xrightarrow{g} kC_p \oplus k \xrightarrow{h} kC_p \to k \to 0  \dots,
   \end{equation*}
   with $k$ in cohomological degree $0$. Again, one can check that $S$ satisfies the conditions of \Cref{about Psi(1)}, with maps 
\[\begin{tikzcd}
	{\unit } & \cdots & 0 & 0 & k & 0 & \cdots \\
	S & \cdots & {kC_p \oplus k} & {kC_p} & k & 0 & \cdots
	\arrow["\zeta"', from=1-1, to=2-1]
	\arrow["{=}"{description}, draw=none, from=1-1, to=2-2]
	\arrow[from=1-2, to=1-3]
	\arrow[from=1-3, to=1-4]
	\arrow[from=1-3, to=2-3]
	\arrow[from=1-4, to=1-5]
	\arrow[from=1-4, to=2-4]
	\arrow[from=1-5, to=1-6]
	\arrow["1", from=1-5, to=2-5]
	\arrow[from=1-6, to=1-7]
	\arrow[from=1-6, to=2-6]
	\arrow["f"', from=2-2, to=2-3]
	\arrow["g"', from=2-3, to=2-4]
	\arrow["h"', from=2-4, to=2-5]
	\arrow[from=2-5, to=2-6]
	\arrow[from=2-6, to=2-7]
\end{tikzcd}\]
and 
\[\begin{tikzcd}
	{\Psi^G(S) } & \cdots & k & k & k & 0 & \cdots \\
	{\unit  } & \cdots & 0 & 0 & k & 0 & \cdots
	\arrow["\sigma"', from=1-1, to=2-1]
	\arrow["{=}"{description}, draw=none, from=1-1, to=2-2]
	\arrow[from=1-2, to=1-3]
	\arrow["0", from=1-3, to=1-4]
	\arrow[from=1-3, to=2-3]
	\arrow["0", from=1-4, to=1-5]
	\arrow[from=1-4, to=2-4]
	\arrow[from=1-5, to=1-6]
	\arrow["1", from=1-5, to=2-5]
	\arrow[from=1-6, to=1-7]
	\arrow[from=1-6, to=2-6]
	\arrow[from=2-2, to=2-3]
	\arrow[from=2-3, to=2-4]
	\arrow[from=2-4, to=2-5]
	\arrow[from=2-5, to=2-6]
	\arrow[from=2-6, to=2-7]
\end{tikzcd}\]
Observe that $S$ is not compact in $\DPerm(C_p;k)$ --- if it were, the object $\Psi^{C_p}(S)$ would be compact in $\D(k)$ as well. By \Cref{residue regularity at closed points}, this implies that $\kappa(\mc M(C_p))$ is not compact in the stalk category $\DPerm(C_p;k)_{\mc M(C_p)}$. Thus, $\DPerm(C_p;k)$ is \emph{not} residually regular at the closed point $\mc M(C_p)$. 
\end{enumerate}
\end{example}

\begin{lemma}\label{inflation}
    Under \Cref{finite group hypothesis}, let $N$ be a normal subgroup of $G$. For any object $x \in \DPerm(G/N;k)^c$, inflation $\Infl^{G/N}_G$ along the quotient map $G \to G/N$ induces a fully faithful functor 
    \begin{equation*}
        \frac{\DPerm(G/N;k)}{\Loc\left(\langle x\rangle\right)} \longrightarrow \frac{\DPerm(G;k)}{\Loc\left(\langle \Infl^{G/N}_G(x)\rangle\right)}.
    \end{equation*}
\end{lemma}
\begin{proof}
    Note that inflation is a fully faithful functor $\DPerm(G/N;k) \to \DPerm(G;k)$ by \cite[Corollary 5.16]{TT-perm}. Now the result follows by an application of \cite[Lemma 5.1]{SpecArtinMotives} using the Thomason subset $\supp_{\mc T(G/N)}(x)$. 
\end{proof}

\begin{proposition}\label{inflation induces an equivalence}
    Under \Cref{finite group hypothesis}, suppose that $G$ is a p-group and let $N$ denote its Frattini subgroup. Then inflation along $G \to G/N$ induces an equivalence of stalk categories: 
    \begin{equation*}
        \DPerm(G/N; k)_{\mc M(G/N)} \xrightarrow{\cong} \DPerm(G; k)_{\mc M(G)}.
    \end{equation*}
\end{proposition}
\begin{proof}
    Consider the object 
    \begin{equation*}
        x = \bigoplus_{K \lneq G/N} k((G/N)/K) \in \mc T(G/N)^c.
    \end{equation*}
    Note that $x$ generates $\mc M(G/N)$ as a tt-ideal (see \cite[Example 7.30]{TT-perm}) and the image of $x$ under inflation is exactly 
    \begin{equation*}
        \Infl^{G/N}_G(x) = \bigoplus_{N \leq H \lneq G} k(G/H).
    \end{equation*}
    Moreover, as $N$ is the Frattini subgroup, we have an equality of tt-ideals in $\DPerm(G;k)$:
    \begin{equation*}
       \mc M(G) = \langle k(G/H) \mid H \lneq G \rangle =  \langle k(G/H)  \mid N \leq H \lneq G  \rangle. 
    \end{equation*}
    After all, any proper subgroup $K \lneq G$ is contained in a maximal subgroup $H$. The inclusion $K \subseteq H$ implies that $\supp(k(G/K)) \subseteq \supp(k(G/H))$ by \cite[Proposition 4.7]{TT-perm} and therefore $k(G/K) \in \langle k(G/H)\rangle$. Since $H$ is maximal, we have $N\subseteq H$ and thus 
    \begin{equation*}
        k(G/K)  \in \langle k(G/H)  \mid N \leq H \lneq G  \rangle.
    \end{equation*}
    
    So, applying \Cref{inflation} with this object $x$ implies exactly that inflation $\Infl^{G/N}_N$ induces a fully faithful geometric tt-functor
    \begin{equation*}
        \DPerm(G/N; k)_{\mc M(G/N)} \longrightarrow \DPerm(G; k)_{\mc M(G)}.
    \end{equation*}
    As both of these categories are generated by their unit, we conclude that the functor is an equivalence.     
\end{proof}

\begin{recollection}
    Let $G = C_{p^n}$, for $n\geq 1$. By \cite[Proposition 8.3]{TT-perm}, the spectrum of $\DPerm(G;k)^c$ has $2n+1$ points arranged in a zigzag pattern with specialisations `pointing up':
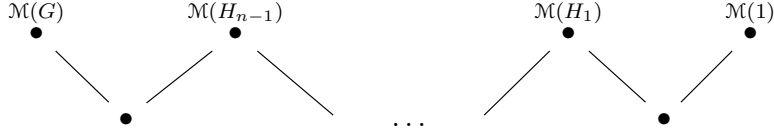
\[\begin{tikzcd}[column sep=1em]
	{\overset{\mc M(G)}{\bullet}} && {\overset{\mc M(H_{n-1})}{\bullet}} &&&& {\overset{\mc M(H_{1})}{\bullet}} && {\overset{\mc M(1)}{\bullet}} \\
	& \bullet && {} & \dots & {} && \bullet
	\arrow[no head, from=1-1, to=2-2]
	\arrow[no head, from=1-3, to=2-4]
	\arrow[no head, from=1-7, to=2-6]
	\arrow[no head, from=1-7, to=2-8]
	\arrow[no head, from=2-2, to=1-3]
	\arrow[no head, from=2-8, to=1-9]
\end{tikzcd}\]
Here $H_i$ denotes the unique subgroup of $G$ of order $p^i$. 
\end{recollection}

\begin{proposition}\label{residue regularity of C_p^n}
    Let $k$ be a field of characteristic $2$. The category $\DPerm(C_{2^n};k)$ is residually regular.
\end{proposition}
\begin{proof}
    We prove this by induction on $n$. In the base case, $\Spc(\mc T(C_2)^c)$ has three points. The rightmost point is $\mc M(1)$ which is always residually regular by \Cref{residue regularity of KInj}. The leftmost point is $\mc M(C_2)$ and it is residually regular as observed in \Cref{cyclic group of order p}. Finally, the stalk category at the bottom point is precisely $\Stab(kC_2)$, which is a tt-field, and therefore we also have residual regularity there by \Cref{tt fields are regular}. 
    
   Now suppose that $n > 1$. The induction hypothesis states that $\mc T(C_{2^{n-1}})$ is residually regular at every point. The restriction functor $\mc T(C_{2^n}) \to \mc T(C_{2^{n-1}})$ catches the rightmost $2n-1$ points of $\Spc(\mc T(C_{2^{n}})^c)$, see \cite[Lemma 8.5]{TT-perm}. By \Cref{restriction to subgroups}, we know that restriction `reflects' residual regularity and thus these $2n-1$ points are all residually regular. The remaining two points are taken care of by \Cref{inflation induces an equivalence}: The stalk category of $\mc T(C_{2^n})$ at $\mc M(C_{2^n})$ is equivalent to the stalk category of $\DPerm(C_{2})$ at $\mc M(C_2)$, of which all points are residually regular by the base case. 
\end{proof}

\begin{remark}
    For $p$ odd, we can push through the same reasoning to find that $\DPerm(C_{p^n};k)$ is residually regular at all points on the bottom, as well as the rightmost closed point $\mc M(1)$, but not at any of the closed points $\mc M(H)$ for $1 \neq H \leq G$.
\end{remark}

\begin{example}\label{klein four}
    Let $E = C_2 \times C_2$ be the Klein four-group. Then $\Spc(\mc T(E)^c)$ has five closed points. As we know, $\mc T(E)$ is certainly residually regular at $\mc M(1)$ by \Cref{residue regularity of KInj}. The closed points $\mc M(N_0), \mc M(N_1)$ and $\mc M(N_{\infty})$, corresponding to the three cyclic subgroups $N_0, N_1$ and $N_{\infty}$ of $E$, each lie in the image of restriction to the respective subgroup: $$\mc M(N_{\nu}) \in \im\left(\varphi_{\Res^G_{N_{\nu}}}\right), \quad \nu =0, 1, \infty.$$ Since $N_{\nu} \cong C_2$, the category $\mc T(N_{\nu})$ is residually regular by  \Cref{residue regularity of C_p^n}. It follows, by \Cref{restriction to subgroups}, that $\mc T(E)$ is also residually regular at $\mc M(N_{\nu})$ for $\nu =0, 1, \infty$. 

For the remaining closed point $\mc M(E)$, we will use \Cref{about Psi(1)} once more. We identify the rings $kE \cong R\vcentcolon= k[x, y]/(x^2, y^2)$ and under this identification we have
\begin{align*}
    & k(E/N_0) \cong M_0 \vcentcolon= \frac{R}{(x)} , \\
    & k(E/N_1) \cong M_1 \vcentcolon= \frac{R}{(y)}, \\
    & k(E/N_{\infty}) \cong M_{\infty} \vcentcolon= \frac{R}{(x+y)}.
\end{align*}
Then consider the following eventually left periodic complex in $\DPerm(E;k)$: 
\begin{equation*}
 S = \dots \xrightarrow{g} M_0 \oplus M_1 \oplus M_{\infty} \xrightarrow{f} R \oplus k \oplus k \xrightarrow{g} M_0 \oplus M_1 \oplus M_{\infty} \xrightarrow{h} R \to k \to 0 \dots
\end{equation*}
with $k$ in cohomological degree $0$ and maps given by 
\begin{align*}
    & h = \begin{pmatrix} x & y & x+y \end{pmatrix} \\
    & g = \begin{pmatrix} 1 & 0 & y \\ 1 & x & 0 \\ 1 & x & y \end{pmatrix} , \\ 
    & f = \begin{pmatrix} x & y & x+y \\ 1 & 0 & 1 \\ 0 & 1 & 1 \end{pmatrix} .
\end{align*}
Applying \Cref{about Psi(1)} with the maps 
\[\begin{tikzcd}
	{\unit } & \cdots & 0 & 0 & k & 0 & \cdots \\
	S & \cdots & {M_0 \oplus M_1 \oplus M_{\infty}} & R & k & 0 & \cdots
	\arrow["\zeta"', from=1-1, to=2-1]
	\arrow["{=}"{description}, draw=none, from=1-1, to=2-2]
	\arrow[from=1-2, to=1-3]
	\arrow[from=1-3, to=1-4]
	\arrow[from=1-3, to=2-3]
	\arrow[from=1-4, to=1-5]
	\arrow[from=1-4, to=2-4]
	\arrow[from=1-5, to=1-6]
	\arrow["1", from=1-5, to=2-5]
	\arrow[from=1-6, to=1-7]
	\arrow[from=1-6, to=2-6]
	\arrow[from=2-2, to=2-3]
	\arrow[from=2-3, to=2-4]
	\arrow[from=2-4, to=2-5]
	\arrow[from=2-5, to=2-6]
	\arrow[from=2-6, to=2-7]
\end{tikzcd}\]
and 
\[\begin{tikzcd}
	{\Psi^G(S) } & \cdots & 0 & 0 & k & 0 & \cdots \\
	{\unit } & \cdots & 0 & 0 & k & 0 & \cdots
	\arrow["\sigma"', from=1-1, to=2-1]
	\arrow["{=}"{description}, draw=none, from=1-1, to=2-2]
	\arrow[from=1-2, to=1-3]
	\arrow[from=1-3, to=1-4]
	\arrow[from=1-3, to=2-3]
	\arrow[from=1-4, to=1-5]
	\arrow[from=1-4, to=2-4]
	\arrow[from=1-5, to=1-6]
	\arrow["1", from=1-5, to=2-5]
	\arrow[from=1-6, to=1-7]
	\arrow[from=1-6, to=2-6]
	\arrow[from=2-2, to=2-3]
	\arrow[from=2-3, to=2-4]
	\arrow[from=2-4, to=2-5]
	\arrow[from=2-5, to=2-6]
	\arrow[from=2-6, to=2-7]
\end{tikzcd}\]
we find that $\kappa(\mc M(E)) \cong S$. Note that $S$ is not compact in $\mc T(E)$ --- if it were then $\Psi^{E}(S)$ would be compact in $\D(k)$ as well. By \Cref{residue regularity at closed points}, this implies that $\kappa(\mc M(E))$ is not compact in the stalk category $\mc T(E)_{\mc M(E)}$. Thus, $\mc T(E)$ is \emph{not} residually regular at the closed point $\mc M(E)$. 
\end{example}

The following purely group theoretical result is standard. 
\begin{lemma}\label{group theory fact}
    If $G$ is a finite $2$-group, then at least one of the following things are true: 
    \begin{itemize}
        \item $G$ is cyclic. 
        \item $G$ has a subgroup isomorphic to $C_2 \times C_2$.
        \item $G$ has a subgroup isomorphic to $Q_8$. 
    \end{itemize}
\end{lemma}
\begin{proof}
    If $G$ does not contain a subgroup isomorphic to $C_2 \times C_2$, then $G$ must have a unique subgroup of order $2$ (since its centre is nontrivial). Then, by application of \cite[Theorem 12.5.2]{hallgrouptheory} (originally in \cite[Section 105]{burnsidegrouptheory}), the group $G$ must either be cyclic or be a generalised quaternion group. In the latter case, $G$ certainly contains a subgroup isomorphic to $Q_8$. 
\end{proof}

\begin{theorem}\label{modular rep theorem}
    Let $G$ be a finite group and $k$ a field of positive characteristic $p$ dividing the order of $G$. Then $\DPerm(G;k)$ is residually regular if and only if $p=2$ and all Sylow $p$-subgroups of $G$ are cyclic. 
\end{theorem}
\begin{proof}
    For any $p$-Sylow subgroup $P$ of $G$, the map $$\varphi_{\Res^G_P} \colon \Spc(\mc T(P)^c) \to \Spc(\mc T(G)^c),$$ is surjective (see \cite[Proposition 3.8]{TT-perm}). By \Cref{restriction to subgroups}, this means that the residual regularity of $\mc T(G)$ is governed by the residual regularity of $\mc T(P)$. So we may reduce to the case where $G$ is a $p$-group. 

   First, suppose $p =2$. If $G$ is cyclic, i.e.\@ $G \cong C_{2^n}$, then we already know that $\mc T(G)$ is residually regular by \Cref{residue regularity of C_p^n}. So, suppose that $G$ is a non-cyclic $2$-group. By \Cref{group theory fact}, $G$ must have a subgroup $H$ with $H \cong C_2 \times C_2$ or $H \cong Q_8$. In \Cref{klein four}, we found that $\mc T(C_2 \times C_2)$ is not residually regular at its closed point $\mc M(C_2 \times C_2)$. Now, note that $Q_8/N \cong C_2 \times C_2$ for the Frattini subgroup $N$ of $Q_8$ and therefore \Cref{inflation induces an equivalence} implies that 
   \begin{equation*}
       \mc T(C_2 \times C_2)_{\mc M(C_2 \times C_2)} \cong \mc T(Q_8)_{\mc M(Q_8)}. 
   \end{equation*}
   Therefore, $\mc T(Q_8)$ is not residually regular at its closed point $\mc M(Q_8)$. We conclude that the category $\mc T(H)$ is not residually regular and so \Cref{restriction to subgroups} implies that $\mc T(G)$ cannot be residually regular either.

   Finally, when $p$ is odd, we simply note that $G$ has a subgroup isomorphic to $C_p$, which is not residually regular as noted in \Cref{cyclic group of order p}. So again we conclude by \Cref{restriction to subgroups}. 
\end{proof}

\begin{remark}
    Suppose that $G$ and $k$ are such that $\DPerm(G;k)$ is not residually regular. Tracing through the proof of \Cref{modular rep theorem}, one can show that the point(s) at which $\DPerm(G;k)$ is not residually regular are all located in the complement of the cohomological open $V_G$. As mentioned in \Cref{conjecture KInj is regular}, the expectation is that all points in the cohomological open, which corresponds to the spectrum of $\K\Inj(kG)$, are residually regular.
\end{remark}

\printbibliography

@article {Balmer-Favi,
    AUTHOR = {Balmer, Paul and Favi, Giordano},
     TITLE = {Generalized tensor idempotents and the telescope conjecture},
   JOURNAL = {Proc. Lond. Math. Soc. (3)},
  FJOURNAL = {Proceedings of the London Mathematical Society. Third Series},
    VOLUME = {102},
      YEAR = {2011},
    NUMBER = {6},
     PAGES = {1161--1185},
      ISSN = {0024-6115},
   MRCLASS = {18E30 (14F05 55P60)},
  MRNUMBER = {2806103},
MRREVIEWER = {Fernando Muro},
       DOI = {10.1112/plms/pdq050},
       URL = {https://doi.org/10.1112/plms/pdq050},
}

@article {Rouquier,
    AUTHOR = {Rouquier, Rapha\"{e}l},
     TITLE = {Dimensions of triangulated categories},
   JOURNAL = {J. K-Theory},
  FJOURNAL = {Journal of K-Theory. K-Theory and its Applications in Algebra,
              Geometry, Analysis \& Topology},
    VOLUME = {1},
      YEAR = {2008},
    NUMBER = {2},
     PAGES = {193--256},
      ISSN = {1865-2433},
   MRCLASS = {18E30 (14F05)},
  MRNUMBER = {2434186},
MRREVIEWER = {Ioannis Emmanouil},
       DOI = {10.1017/is007011012jkt010},
       URL = {https://doi.org/10.1017/is007011012jkt010},
}

@article{BKSruminations,
   title={Tensor-triangular fields: ruminations},
   volume={25},
   ISSN={1420-9020},
   url={http://dx.doi.org/10.1007/s00029-019-0454-2},
   DOI={10.1007/s00029-019-0454-2},
   number={1},
   journal = {Selecta Math.},
   fjournal={Selecta Mathematica},
   publisher={Springer Science and Business Media LLC},
   author={Balmer, Paul and Krause, Henning and Stevenson, Greg},
   year={2019}
   }

@article{Balmernilpotence,
   title={Nilpotence theorems via homological residue fields},
   volume={2},
   ISSN={2576-7658},
   url={http://dx.doi.org/10.2140/tunis.2020.2.359},
   DOI={10.2140/tunis.2020.2.359},
   number={2},
   journal = {Tunis. J. Math.},
   fjournal={Tunisian Journal of Mathematics},
   publisher={Mathematical Sciences Publishers},
   author={Balmer, Paul},
   year={2020},
   pages={359–378} }

@article{BalmerCameron,
   title={Computing homological residue fields in algebra and topology},
   volume={149},
   ISSN={1088-6826},
   url={http://dx.doi.org/10.1090/proc/15412},
   DOI={10.1090/proc/15412},
   number={8},
   journal = {Proc. Amer. Math. Soc.},
   fjournal={Proceedings of the American Mathematical Society},
   publisher={American Mathematical Society (AMS)},
   author={Balmer, Paul and Cameron, James},
   year={2021},
   pages={3177–3185} }

@article{Sandersetale,
   title={A characterization of finite etale morphisms in tensor triangular geometry},
   volume={Volume 6},
   ISSN={2491-6765},
   url={http://dx.doi.org/10.46298/epiga.2022.volume6.7641},
   DOI={10.46298/epiga.2022.volume6.7641},
   journal={Épijournal de Géométrie Algébrique},
   publisher={Centre pour la Communication Scientifique Directe (CCSD)},
   author={Sanders, Beren},
   year={2022}
   }

@article{BalmerQuillenstrat,
   title={Separable extensions in tensor-triangular geometry and generalized Quillen stratification},
   volume={49},
   ISSN={1873-2151},
   url={http://dx.doi.org/10.24033/asens.2298},
   DOI={10.24033/asens.2298},
   number={4},
   journal ={Ann. Sci. Éc. Norm. Supér.},
fjournal={Annales scientifiques de l’École normale supérieure},
   publisher={Societe Mathematique de France},
   author={Balmer, Paul},
   year={2016},
   pages={907–925} }

@article {TT-perm,
    AUTHOR = {Balmer, Paul and Gallauer, Martin},
     TITLE = {The geometry of permutation modules},
   JOURNAL = {Invent. Math.},
  FJOURNAL = {Inventiones Mathematicae},
    VOLUME = {241},
      YEAR = {2025},
    NUMBER = {3},
     PAGES = {841--928},
      ISSN = {0020-9910,1432-1297},
   MRCLASS = {20C20 (13A02 18G80 20J06)},
  MRNUMBER = {4946248},
       DOI = {10.1007/s00222-025-01350-z},
       URL = {https://doi.org/10.1007/s00222-025-01350-z},
}

@article {Bregje,
    AUTHOR = {Pauwels, Bregje},
     TITLE = {Quasi-{G}alois theory in symmetric monoidal categories},
   JOURNAL = {Algebra Number Theory},
  FJOURNAL = {Algebra \& Number Theory},
    VOLUME = {11},
      YEAR = {2017},
    NUMBER = {8},
     PAGES = {1891--1920},
      ISSN = {1937-0652,1944-7833},
   MRCLASS = {18Bxx (16Gxx 18Gxx)},
  MRNUMBER = {3720934},
MRREVIEWER = {Toma\ Albu},
       DOI = {10.2140/ant.2017.11.1891},
       URL = {https://doi.org/10.2140/ant.2017.11.1891},
}

@article {BHS,
    AUTHOR = {Barthel, Tobias and Heard, Drew and Sanders, Beren},
     TITLE = {Stratification in tensor triangular geometry with applications
              to spectral {M}ackey functors},
   JOURNAL = {Camb. J. Math.},
  FJOURNAL = {Cambridge Journal of Mathematics},
    VOLUME = {11},
      YEAR = {2023},
    NUMBER = {4},
     PAGES = {829--915},
      ISSN = {2168-0930,2168-0949},
   MRCLASS = {18G80 (14F08 18F99 55P42 55P91 55U35)},
  MRNUMBER = {4650265},
       DOI = {10.4310/cjm.2023.v11.n4.a2},
       URL = {https://doi.org/10.4310/cjm.2023.v11.n4.a2},
}

@book {spectralbook,
    AUTHOR = {Dickmann, Max and Schwartz, Niels and Tressl, Marcus},
     TITLE = {Spectral spaces},
    SERIES = {New Mathematical Monographs},
    VOLUME = {35},
 PUBLISHER = {Cambridge University Press, Cambridge},
      YEAR = {2019},
      ISBN = {978-1-107-14672-3},
   MRCLASS = {54-02 (06B30 06D50 06F30 18B35 54B30 54H10)},
  MRNUMBER = {3929704},
MRREVIEWER = {Tomasz\ Kubiak},
       DOI = {10.1017/9781316543870},
       URL = {https://doi.org/10.1017/9781316543870},
}

@article {supportsandfiltrations,
    AUTHOR = {Balmer, Paul},
     TITLE = {Supports and filtrations in algebraic geometry and modular
              representation theory},
   JOURNAL = {Amer. J. Math.},
  FJOURNAL = {American Journal of Mathematics},
    VOLUME = {129},
      YEAR = {2007},
    NUMBER = {5},
     PAGES = {1227--1250},
      ISSN = {0002-9327,1080-6377},
   MRCLASS = {18E30 (20C20)},
  MRNUMBER = {2354319},
MRREVIEWER = {Amnon\ Neeman},
       DOI = {10.1353/ajm.2007.0030},
       URL = {https://doi.org/10.1353/ajm.2007.0030},
}

@misc{kostas,
      title={Intrinsic homological algebra for triangulated categories}, 
      author={Panagiotis Kostas and Chrysostomos Psaroudakis and Jorge Vitória},
      year={2025},
      eprint={2512.18417},
      archivePrefix={arXiv},
      primaryClass={math.RT},
      url={https://arxiv.org/abs/2512.18417}, 
}

@article {Orlov,
    AUTHOR = {Orlov, Dmitri},
     TITLE = {Smooth and proper noncommutative schemes and gluing of {DG}
              categories},
   JOURNAL = {Adv. Math.},
  FJOURNAL = {Advances in Mathematics},
    VOLUME = {302},
      YEAR = {2016},
     PAGES = {59--105},
      ISSN = {0001-8708,1090-2082},
   MRCLASS = {14F05 (16E45 18E30)},
  MRNUMBER = {3545926},
MRREVIEWER = {Shintarou\ Yanagida},
       DOI = {10.1016/j.aim.2016.07.014},
       URL = {https://doi.org/10.1016/j.aim.2016.07.014},
}

@misc{BIKP,
      title={Locally dualisable modular representations and local regularity}, 
      author={Dave Benson and Srikanth Iyengar and Henning Krause and Julia Pevtsova},
      year={2024},
      eprint={2404.14672},
      archivePrefix={arXiv},
      primaryClass={math.RT},
      url={https://arxiv.org/abs/2404.14672}, 
}

@article {GregGreenlees,
    AUTHOR = {Greenlees, J. P. C. and Stevenson, Greg},
     TITLE = {Morita theory and singularity categories},
   JOURNAL = {Adv. Math.},
  FJOURNAL = {Advances in Mathematics},
    VOLUME = {365},
      YEAR = {2020},
     PAGES = {107055, 51},
      ISSN = {0001-8708,1090-2082},
   MRCLASS = {55P43 (13D09 16E45 20J06 55P62)},
  MRNUMBER = {4065715},
MRREVIEWER = {Steffen\ Sagave},
       DOI = {10.1016/j.aim.2020.107055},
       URL = {https://doi.org/10.1016/j.aim.2020.107055},
}

@article {Neeman2026,
    AUTHOR = {Neeman, Amnon},
     TITLE = {Triangulated categories with a single compact generator, and
              two {B}rown representability theorems},
   JOURNAL = {Invent. Math.},
  FJOURNAL = {Inventiones Mathematicae},
    VOLUME = {244},
      YEAR = {2026},
    NUMBER = {2},
     PAGES = {531--616},
      ISSN = {0020-9910,1432-1297},
   MRCLASS = {18E30 (14F05)},
  MRNUMBER = {5012578},
       DOI = {10.1007/s00222-025-01401-5},
       URL = {https://doi.org/10.1007/s00222-025-01401-5},
}

@misc{BarwickLawson,
      title={Regularity of structured ring spectra and localization in {K}-theory}, 
      author={Clark Barwick and Tyler Lawson},
      year={2014},
      eprint={1402.6038},
      archivePrefix={arXiv},
      primaryClass={math.KT},
      url={https://arxiv.org/abs/1402.6038}, 
}

@article {tt-rings,
    AUTHOR = {Balmer, Paul},
     TITLE = {Separability and triangulated categories},
   JOURNAL = {Adv. Math.},
  FJOURNAL = {Advances in Mathematics},
    VOLUME = {226},
      YEAR = {2011},
    NUMBER = {5},
     PAGES = {4352--4372},
      ISSN = {0001-8708,1090-2082},
   MRCLASS = {18E30 (13B40 16H05 18C20)},
  MRNUMBER = {2770453},
MRREVIEWER = {Matthias\ K\"{u}nzer},
       DOI = {10.1016/j.aim.2010.12.003},
       URL = {https://doi.org/10.1016/j.aim.2010.12.003},
}

@inproceedings{ICMpaper,
    AUTHOR = {Balmer, Paul},
     TITLE = {Tensor triangular geometry},
 BOOKTITLE = {Proceedings of the {I}nternational {C}ongress of
              {M}athematicians. {V}olume {II}},
     PAGES = {85--112},
 PUBLISHER = {Hindustan Book Agency, New Delhi},
      YEAR = {2010},
   MRCLASS = {18E30 (14F05 19G12 55P42)},
  MRNUMBER = {2827786},
MRREVIEWER = {Sunil\ K.\ Chebolu},
}

@article {restrict,
    AUTHOR = {Balmer, Paul and Dell'Ambrogio, Ivo and Sanders, Beren},
     TITLE = {Restriction to finite-index subgroups as \'{e}tale extensions
              in topology, {KK}-theory and geometry},
   JOURNAL = {Algebr. Geom. Topol.},
  FJOURNAL = {Algebraic \& Geometric Topology},
    VOLUME = {15},
      YEAR = {2015},
    NUMBER = {5},
     PAGES = {3025--3047},
      ISSN = {1472-2747,1472-2739},
   MRCLASS = {55P91 (13D09 14F05 19K35 19L47)},
  MRNUMBER = {3426702},
MRREVIEWER = {Andr\'{e}\ G.\ Henriques},
       DOI = {10.2140/agt.2015.15.3025},
       URL = {https://doi.org/10.2140/agt.2015.15.3025},
}

@article{BHScomp,
   title={Stratification and the comparison between homological and tensor triangular support},
   volume={74},
   ISSN={1464-3847},
   url={http://dx.doi.org/10.1093/qmath/haac040},
   DOI={10.1093/qmath/haac040},
   number={2},
   journal ={Q. J. Math.},
   fjournal={The Quarterly Journal of Mathematics},
   publisher={Oxford University Press (OUP)},
   author={Barthel, Tobias and Heard, Drew and Sanders, Beren},
   year={2023},
   pages={747–766} }

@article {finitedegree,
    AUTHOR = {Balmer, Paul},
     TITLE = {Splitting tower and degree of tt-rings},
   JOURNAL = {Algebra Number Theory},
  FJOURNAL = {Algebra \& Number Theory},
    VOLUME = {8},
      YEAR = {2014},
    NUMBER = {3},
     PAGES = {767--779},
      ISSN = {1937-0652,1944-7833},
   MRCLASS = {18E30 (20J05)},
  MRNUMBER = {3218809},
MRREVIEWER = {Frank\ Lucas\ Wolcott},
       DOI = {10.2140/ant.2014.8.767},
       URL = {https://doi.org/10.2140/ant.2014.8.767},
}

@article {bigsupport,
    AUTHOR = {Balmer, Paul},
     TITLE = {Homological support of big objects in tensor-triangulated
              categories},
   JOURNAL = {J. \'{E}c. polytech. Math.},
  FJOURNAL = {Journal de l'\'{E}cole polytechnique. Math\'{e}matiques},
    VOLUME = {7},
      YEAR = {2020},
     PAGES = {1069--1088},
      ISSN = {2429-7100,2270-518X},
   MRCLASS = {18G80 (18M05 20J05 55U35)},
  MRNUMBER = {4136434},
MRREVIEWER = {Bo\ Lu},
       DOI = {10.5802/jep.135},
       URL = {https://doi.org/10.5802/jep.135},
}

@article {Neemanrepresentability,
    AUTHOR = {Neeman, Amnon},
     TITLE = {The {G}rothendieck duality theorem via {B}ousfield's
              techniques and {B}rown representability},
   JOURNAL = {J. Amer. Math. Soc.},
  FJOURNAL = {Journal of the American Mathematical Society},
    VOLUME = {9},
      YEAR = {1996},
    NUMBER = {1},
     PAGES = {205--236},
      ISSN = {0894-0347,1088-6834},
   MRCLASS = {18E30 (14F05)},
  MRNUMBER = {1308405},
MRREVIEWER = {Luca\ Barbieri Viale},
       DOI = {10.1090/S0894-0347-96-00174-9},
       URL = {https://doi.org/10.1090/S0894-0347-96-00174-9},
}

@article {Krause,
    AUTHOR = {Krause, Henning},
     TITLE = {Smashing subcategories and the telescope conjecture---an
              algebraic approach},
   JOURNAL = {Invent. Math.},
  FJOURNAL = {Inventiones Mathematicae},
    VOLUME = {139},
      YEAR = {2000},
    NUMBER = {1},
     PAGES = {99--133},
      ISSN = {0020-9910,1432-1297},
   MRCLASS = {55P42 (18E30 55P60 55P65)},
  MRNUMBER = {1728877},
MRREVIEWER = {Paul\ G.\ Goerss},
       DOI = {10.1007/s002229900022},
       URL = {https://doi.org/10.1007/s002229900022},
}

@article {SpecArtinMotives,
    AUTHOR = {Balmer, Paul and Gallauer, Martin},
     TITLE = {The spectrum of {A}rtin motives},
   JOURNAL = {Trans. Amer. Math. Soc.},
  FJOURNAL = {Transactions of the American Mathematical Society},
    VOLUME = {378},
      YEAR = {2025},
    NUMBER = {3},
     PAGES = {1733--1754},
      ISSN = {0002-9947,1088-6850},
   MRCLASS = {14F42 (18F99 18G90 20C20)},
  MRNUMBER = {4866349},
MRREVIEWER = {Oliver\ R\"{o}ndigs},
       DOI = {10.1090/tran/9306},
       URL = {https://doi.org/10.1090/tran/9306},
}

@incollection {homotopyhandbook,
    AUTHOR = {Balmer, Paul},
     TITLE = {A guide to tensor-triangular classification},
 BOOKTITLE = {Handbook of homotopy theory},
    SERIES = {CRC Press}, %/Chapman Hall Handb. Math. Ser.
     PAGES = {145--162},
 PUBLISHER = {CRC Press, Boca Raton, FL},
      YEAR = {2020},
      ISBN = {978-0-815-36970-7},
   MRCLASS = {55U35 (18G80 55P42)},
  MRNUMBER = {4197984},
MRREVIEWER = {Philippe\ Gaucher},
}

@incollection {Stevensontour,
    AUTHOR = {Stevenson, Greg},
     TITLE = {A tour of support theory for triangulated categories through
              tensor triangular geometry},
 BOOKTITLE = {Building bridges between algebra and topology},
    SERIES = {Adv. Courses Math. CRM Barcelona},
     PAGES = {63--101},
 PUBLISHER = {Birkh\"{a}user/Springer, Cham},
      YEAR = {2018},
   MRCLASS = {18E30 (55P60)},
  MRNUMBER = {3793858},
MRREVIEWER = {Beren\ Sanders},
}

@article {BalmerOG,
    AUTHOR = {Balmer, Paul},
     TITLE = {The spectrum of prime ideals in tensor triangulated
              categories},
   JOURNAL = {J. Reine Angew. Math.},
  FJOURNAL = {Journal f\"{u}r die Reine und Angewandte Mathematik. [Crelle's
              Journal]},
    VOLUME = {588},
      YEAR = {2005},
     PAGES = {149--168},
      ISSN = {0075-4102,1435-5345},
   MRCLASS = {18E30 (55P99)},
  MRNUMBER = {2196732},
MRREVIEWER = {Amnon\ Neeman},
       DOI = {10.1515/crll.2005.2005.588.149},
       URL = {https://doi.org/10.1515/crll.2005.2005.588.149},
}

@incollection {BGconnections,
    AUTHOR = {Balmer, Paul and Gallauer, Martin},
     TITLE = {Permutation modules, {M}ackey functors, and {A}rtin motives},
 BOOKTITLE = {Representations of algebras and related structures},
    SERIES = {EMS Ser. Congr. Rep.},
     PAGES = {37--75},
 PUBLISHER = {EMS Press, Berlin},
      YEAR = {2023},
      ISBN = {978-3-98547-054-9},
   MRCLASS = {14F08 (12F10 14C15 18G80 20C20)},
  MRNUMBER = {4693637},
}

@article {Stevensonabsflatrings,
    AUTHOR = {Stevenson, Greg},
     TITLE = {Derived categories of absolutely flat rings},
   JOURNAL = {Homology Homotopy Appl.},
  FJOURNAL = {Homology, Homotopy and Applications},
    VOLUME = {16},
      YEAR = {2014},
    NUMBER = {2},
     PAGES = {45--64},
      ISSN = {1532-0073,1532-0081},
   MRCLASS = {13D09},
  MRNUMBER = {3234500},
MRREVIEWER = {Sunil\ K.\ Chebolu},
       DOI = {10.4310/HHA.2014.v16.n2.a3},
       URL = {https://doi.org/10.4310/HHA.2014.v16.n2.a3},
}

@article {infdegttrings,
    AUTHOR = {G\'{o}mez, Juan Omar},
     TITLE = {A family of infinite degree tt-rings},
   JOURNAL = {Bull. Lond. Math. Soc.},
  FJOURNAL = {Bulletin of the London Mathematical Society},
    VOLUME = {56},
      YEAR = {2024},
    NUMBER = {2},
     PAGES = {518--522},
      ISSN = {0024-6093,1469-2120},
   MRCLASS = {18G80},
  MRNUMBER = {4711566},
MRREVIEWER = {Bo\ Lu},
       DOI = {10.1112/blms.12945},
       URL = {https://doi.org/10.1112/blms.12945},
}

@misc{nervesofsteelcounterex,
      title={Geometric Points in Tensor Triangular Geometry}, 
      author={Tobias Barthel and Logan Hyslop and Maxime Ramzi},
      year={2026},
      eprint={2603.25664},
      archivePrefix={arXiv},
      primaryClass={math.AT},
      url={https://arxiv.org/abs/2603.25664}, 
}

@phdthesis{Pauwelsphd,
  author = {Bregje Pauwels},
  title = {Quasi-Galois theory in triangulated categories},
  school = {University of California, Los Angeles},
  year = {2015}
}

@misc{regandsepmonads,
      title={Homological Aspects of Separable Extensions of Triangulated Categories}, 
      author={Miltiadis Karakikes and Panagiotis Kostas},
      year={2026},
      eprint={2604.17996},
      archivePrefix={arXiv},
      primaryClass={math.RT},
      url={https://arxiv.org/abs/2604.17996}, 
}

@article {Krausecompletion,
    AUTHOR = {Krause, Henning},
     TITLE = {Completing perfect complexes},
      NOTE = {With appendices by Tobias Barthel and Bernhard Keller},
   JOURNAL = {Math. Z.},
  FJOURNAL = {Mathematische Zeitschrift},
    VOLUME = {296},
      YEAR = {2020},
    NUMBER = {3-4},
     PAGES = {1387--1427},
      ISSN = {0025-5874,1432-1823},
   MRCLASS = {18G80 (14F08 16E35 55P42)},
  MRNUMBER = {4159834},
MRREVIEWER = {Bo\ Lu},
       DOI = {10.1007/s00209-020-02490-z},
       URL = {https://doi.org/10.1007/s00209-020-02490-z},
}

@misc{singcats,
      title={Bounded $t$-structures, finitistic dimensions, and singularity categories of triangulated categories}, 
      author={Rudradip Biswas and Hongxing Chen and Kabeer Manali Rahul and Chris J. Parker and Junhua Zheng},
      year={2024},
      eprint={2401.00130},
      archivePrefix={arXiv},
      primaryClass={math.RA},
      url={https://arxiv.org/abs/2401.00130}, 
}

@article {BKSOG,
    AUTHOR = {Balmer, Paul and Krause, Henning and Stevenson, Greg},
     TITLE = {The frame of smashing tensor-ideals},
   JOURNAL = {Math. Proc. Cambridge Philos. Soc.},
  FJOURNAL = {Mathematical Proceedings of the Cambridge Philosophical
              Society},
    VOLUME = {168},
      YEAR = {2020},
    NUMBER = {2},
     PAGES = {323--343},
      ISSN = {0305-0041,1469-8064},
   MRCLASS = {18M99 (16D90)},
  MRNUMBER = {4064108},
MRREVIEWER = {Sultan\ Eylem\ Toksoy},
       DOI = {10.1017/s0305004118000725},
       URL = {https://doi.org/10.1017/s0305004118000725},
}

@article {Thomasonclassific,
    AUTHOR = {Thomason, R. W.},
     TITLE = {The classification of triangulated subcategories},
   JOURNAL = {Compositio Math.},
  FJOURNAL = {Compositio Mathematica},
    VOLUME = {105},
      YEAR = {1997},
    NUMBER = {1},
     PAGES = {1--27},
      ISSN = {0010-437X,1570-5846},
   MRCLASS = {18E30 (18F30)},
  MRNUMBER = {1436741},
MRREVIEWER = {Steven\ E.\ Landsburg},
       DOI = {10.1023/A:1017932514274},
       URL = {https://doi.org/10.1023/A:1017932514274},
}

@misc{singcatsepext,
      title={The singularity category of a separable extension}, 
      author={Charalampos Verasdanis},
      year={2026},
      eprint={2605.08868},
      archivePrefix={arXiv},
      primaryClass={math.CT},
      url={https://arxiv.org/abs/2605.08868}, 
}

@book{burnsidegrouptheory,
  title={Theory of groups of finite order (2nd edition)},
  author={William Burnside},
  %isbn={9780387906881},
  %lccn={82000733},
  %series={Graduate Texts in Mathematics},
  %url={https://books.google.com/books?id=PMqb2DppvCsC},
  year={1911},
  publisher={Cambridge University Press, Cambridge}
}

@book{hallgrouptheory,
  title={The theory of groups (2nd edition)},
  author={Hall, Jr., Marshall},
  %isbn={9780387906881},
  %lccn={82000733},
  %series={Graduate Texts in Mathematics},
  %url={https://books.google.com/books?id=PMqb2DppvCsC},
  year={1976},
  publisher={AMS Chelsea Publishing, Providence}
}

\end{document}